\documentclass[10pt,a4paper]{article}
\usepackage[T1]{fontenc}
\usepackage{amsmath,amsthm,amssymb,xcolor,fullpage,enumitem}
\usepackage[hidelinks]{hyperref}
\usepackage{booktabs} 
\usepackage{array}    
\usepackage{pgfplots}
\usepackage{comment}
\usepackage{tikz}
\usetikzlibrary{calc, decorations.pathreplacing}
\pgfplotsset{compat=newest}
\usetikzlibrary{calc,patterns}
\usepgfplotslibrary{fillbetween}

\numberwithin{equation}{section}
\numberwithin{figure}{section}

\newtheorem{thm}{Theorem}[section]
\newtheorem{lemma}[thm]{Lemma}
\newtheorem{prop}[thm]{Proposition}
\newtheorem{cor}[thm]{Corollary}

\theoremstyle{definition}
\newtheorem{rmk}[thm]{Remark}

\theoremstyle{remark}
\newtheorem{example}[thm]{Example}

\DeclareMathOperator{\diag}{diag}
\DeclareMathOperator{\dist}{dist}

\newcommand{\eps}{{\varepsilon}}

\newcommand{\C}{{\mathbb C}}
\newcommand{\R}{{\mathbb R}}

\newcommand{\hf}{{\hat f}}
\newcommand{\hA}{{\hat A}}

\newcommand{\cE}{{\mathcal E}}
\newcommand{\cL}{{\mathcal L}}
\newcommand{\cP}{{\mathcal P}}

\newcommand{\tc}{{\tilde c}}
\newcommand{\tf}{{\tilde f}}
\newcommand{\tA}{{\widetilde A}}
\newcommand{\tT}{{\widetilde T}}
\newcommand{\tcP}{{\widetilde{\mathcal P}}}

\newcommand{\bpi}{{\breve{\pi}}}

\newcommand{\bDelta}{{\breve{\Delta}}}

\newcommand{\bh}{{\breve{h}}}

\newcommand{\bW}{{\breve{W}}}

\newcommand{\ddvarkappa}{{\ddot{\varkappa}}}
\newcommand{\ddnu}{{\ddot{\nu}}}
\newcommand{\ddn}{{\ddot{n}}}
\newcommand{\ddr}{{\ddot{r}}}
\newcommand{\ddx}{{\ddot{x}}}
\newcommand{\ddC}{{\ddot{C}}}
\newcommand{\ddR}{{\ddot{R}}}
\newcommand{\ddcP}{{\ddot{\cP}}}

\begin{document}

\title{Rigorous and effective numerics for Liverani-Saussol-Vaienti maps}
\author{
Alexey Korepanov\thanks{School of Sciences, Great Bay University, Dongguan 523000, PR China.
\url{khumarahn@gmail.com}}
\and
YuTong Wei\thanks{School of Mathematical Sciences, University of Science and Technology of China, Hefei 230026, PR China.
\url{wytustc@mail.ustc.edu.cn}}
\and
Caroline Wormell\thanks{School of Mathematics and Statistics, The University of Sydney, NSW 2006, Australia.
\url{caroline.wormell@sydney.edu.au}}}
\date{1 October 2026}

\maketitle

\begin{abstract}
    We make highly accurate numerical estimates of basic properties of the invariant
    measure for Liverani-Saussol-Vaienti intermittent maps, an archetypal slowly
    mixing dynamical system. We solve the challenge of precise and efficient
    estimation for a map with non-uniform expansion, where the transfer operator
    does not have a spectral gap. We do this using an Abel function, which ``solves''
    the neutral dynamics, and which we can compute accurately via an asymptotic expansion.

    Our work covers both finite and sigma-finite invariant measure cases.
    In particular, we obtain the first practical estimates for the parameter
    far from zero, including the sigma-finite case.
    This opens the door to much deeper numerical study of intermittent
    dynamics than was previously possible.

    A driving motivation for this work is the upcoming proof of optimal rates of
    \emph{memory loss} which works equally for finite and infinite invariant
    measure cases. No prior results on rates of memory loss for dynamical systems
    with infinite invariant measure exist, with a single exception of a recent
    paper by I.Chevyrev and A.K.\ on null recurrent Markov chains.
\end{abstract}

\section{Introduction}

Let $\gamma > 0$ and $T \colon (0,1] \to (0,1]$ be the Liverani-Saussol-Vaienti (LSV)
intermittent map~\cite{LSV99}:
\begin{equation}
    \label{eq:LSV}
    T(x) = \begin{cases}
        x (1 + 2^\gamma x^\gamma) , & x \in (0, 1/2] \\
         2 x - 1, & x \in (1/2,1]
    \end{cases}
\end{equation}
The interest in this transformation comes from physics~\cite{PM80}, where it serves
as a simple model for intermittent transitions between laminar flow and turbulence.
As perhaps the most basic chaotic dynamical system with nonuniform expansion,
it has been studied with great care. We have at our disposal fundamental
results such as decay of correlations~\cite{Y99, G04b},
invariance principles~\cite{MN05,CDKM20}, concentration inequalities~\cite{CCRV09,CG12,GM14},
linear response~\cite{BT16,K16,BS16}, a dynamical version of
renewal theory~\cite{S02,G04a,MT12,T15,T16}, and much more.

For each $\gamma > 0$, this system admits a unique (up to scaling) absolutely continuous
finite or sigma-finite invariant measure. When $\gamma < 1$ this measure is finite,
and when $\gamma \geq 1$ it is infinite. The density of this measure, which we denote
$h(x)$, is smooth away from $0$ (see~\cite[Remark~1.4]{BK24}), and near $0$ behaves asymptotically roughly
as $x^{-\gamma}$. In the setting of these maps, no closed forms are known for the invariant measure or most other related statistical quantities.

\begin{figure}[h!]\label{fig:h}
    \centering
    \begin{tikzpicture}
        \begin{axis}[
            height       = 2in,
            xmin         = 0, xmax = 1.1,
            ymin         = 0, ymax = 27,
            xtick        = {0.0, 0.5, 1.0},
            xticklabels  = {$0$, $1/2$, $1$},
            ytick        = {0, 10, 20, 30},
            axis x line  = bottom,
            axis y line  = left,
            line cap     = round,
            ]
            \addplot[red,thick] coordinates {
                    (0.02,27.9618) (0.03,19.5313) (0.04,15.2344) (0.05,12.6177)
                    (0.06,10.8516) (0.07,9.57621) (0.08,8.61014) (0.10,7.23982)
                    (0.12,6.31078) (0.15,5.36394) (0.18,4.71857) (0.20,4.39001)
                    (0.25,3.78485) (0.30,3.36766) (0.35,3.05986) (0.40,2.82165)
                    (0.45,2.63065) (0.50,2.47327) (0.55,2.34078) (0.60,2.22728)
                    (0.65,2.12863) (0.70,2.04186) (0.80,1.89564) (0.90,1.77644)
                    (1.00,1.67680)
                };
            \node[red,anchor=north west] at (axis cs:0.04,24) {$h(x)$};
        \end{axis}
    \end{tikzpicture}
    \caption{Invariant density $h$ for $\gamma = 95/100$.}
\end{figure}
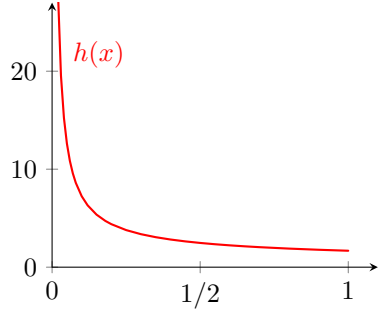

The goal of this paper is a rigorous and accurate numerical estimation of $h$.
This is delicate yet relatively straightforward for uniformly expanding
maps~\cite{L01,GN14,W19,BS20}, but is highly challenging for intermittent maps
due to the nonuniform expansion near $0$.
Some papers attempted to approximate the transfer operator of the induced map
(i.e.\ the first return map to $[1/2,1]$), running into the issue that
this operator is an infinite, slowly-decaying sum over large iterates
of the map $T$, very slow to compute naively.
See~\cite{M10,FMS11,BBD15,NT23}.

We describe an approach which allows efficient estimation of $h$ and its derivatives with rigorous error bounds:
\begin{itemize}
    \item In a complex Bernstein ellipse with foci at $1/2$ and $1$, where we prove that
        $h$ is analytic and approximate it by a finite sum of Chebyshev polynomials
        (Section~\ref{sec:id}, in particular Theorem~\ref{thm:hind}).
    \item On $(0,1]$ using the values in the ellipse and an Euler-Maclaurin formula
        (Section~\ref{sec:h01}).
\end{itemize}

Our approach takes three key steps.
First, we conjugate the left branch of $T$ to a suitable Abel function $A$
(so that $T^n(x) = A^{-1}(A(x)-n)$), for which we have a quickly converging asymptotic
expansion. Thus we can compute orbits near zero. Second, working with the analytic
continuation of $A$, we employ the Euler-Maclaurin formula to accurately
compute the action of the transfer operator (as in~\cite{VW25}).
Finally, we approximate the transfer operator by a finite matrix acting on finitely
many Chebyshev polynomials. In an appropriate Hardy space, the approximation is very tight,
and we approximate $h$ by the principal eigenvector of this matrix. An earlier version of some of these steps appeared in the PhD thesis of C.Wormell~\cite{W20}.

The computational effort of our method grows polynomially with the number of {\it digits} of accuracy. Furthermore, the effort required using our methods does not vary massively with $\gamma$: the web app given below allows computation of properties of the map for a very wide range of $\gamma$ values. This contrasts with previous results, where the effort grew exponentially with the number of digits, and reasonable estimates of statistical properties were in practice restricted to $\gamma \lesssim \tfrac 14$.

A special application and motivation for our work is \emph{memory loss}:
the speed at which the dynamical system forgets its initial state.
It is measured as the rate of decay of $(T^n)_* (\mu - \nu)$ in total variation,
where $\mu$ and $\nu$ are sufficiently regular probability measures and
$(T^n)_*$ is the push-forward operator.
This is a well studied subject~\cite{AHNTV15,BBR19,KL21}, yet
at the moment of writing this manuscript, the literature has no results
for rates of memory loss in ergodic theory with infinite invariant measure
(when $\gamma \geq 1$ in our case).
In the simpler setting of Markov chains there is only the recent paper~\cite{CK25}.

We verify numerically, for a small interval of values of $\gamma$ at a time,
that $h'/h$ and $h''/h$ are monotone on $(0,1]$. This is very complex and technical,
but it allows a relatively simple proof of memory loss with optimal rates, independent
of whether the invariant measure is finite or infinite.
The memory loss itself is a subject of our follow-up paper, in preparation.

Finally, we note that Abel functions have been underutilised in studying intermittent
maps~\cite{P92, R99}, our study suggests their great utility in understanding fine
properties of the dynamical system. For example, it is an immediate byproduct of our
work that the invariant densities are analytic:

\begin{prop}
    \label{prop:anal}
    The invariant density of the LSV map is real-analytic on $(0,1]$ for all $\gamma > 0$.
\end{prop}

\begin{proof}
    The invariant density is analytic on $[1/2,1]$ by Lemma~\ref{lem:h} (see also \cite[Theorem~11.4]{BJ07}). Then, equation~\eqref{eq:Sh} gives analyticity on the full domain.
\end{proof}

\begin{rmk}
    There exist LSV-like maps where the invariant density is exactly
    $h(x) = x^{-\gamma} + (1+x)^{-\gamma}$, see~\cite{T00}.
    With the density known explicitly, some of our analysis becomes
    trivial; although computing e.g.\ the mean return time is still
    challenging.
\end{rmk}

\begin{rmk}
    Our results can be easily generalised to maps of the form
    $f(x) = x \hf(x^\gamma)$ near zero with $\hf$ analytic.
    However, the maps described in the previous remark are not
    of this form, and the asymptotic expansion~\eqref{eq:An}
    of the Abel function would contain many extra terms.
\end{rmk}

The paper is organized as follows: in Section~\ref{sec:Abel} we introduce
and approximate the Abel function; the main result is Theorem~\ref{thm:A}.
In Section~\ref{sec:sums} we estimate the sums that appear in the expression
for the induced transfer operator using the Euler-Maclaurin formula.
In Section~\ref{sec:id} we relate a finite dimensional approximation of
the induced transfer operator to the true one, and approximate
the invariant density $h(x)$ near $[1/2,1]$ by a finite sum of Chebyshev polynomials;
the main result is Theorem~\ref{thm:hind}.
Section~\ref{sec:app} contains technical bounds for our key applications:
approximation of average return time to $[1/2,1]$ and Lyapunov exponent.
In Section~\ref{sec:h01} we extend the estimation of the invariant density
to the full interval $(0,1]$.
In Section~\ref{sec:cond} we set up the verification that $h'/h$ and $h''/h$
are monotone, the conditions needed for the proof of memory loss.
The Appendix collects basic auxiliary proofs.

\subsubsection*{Code and examples}

\begin{example}
    (Section~\ref{sec:app})
    For $\gamma = 95/100$, the average return time to $[1/2,1]$ is
    \begin{align*}
        \mathtt{14.}&\mathtt{0733232200019395292415496996107566098033171694810767816970656955638725907638} \\
                    &\mathtt{2320444598389786058368555531687805455107565467321090680062888709278684158464} \\
                    &\mathtt{2349498540837744246167258662127059795378979037263342948570166140599348141831} \\
                    &\mathtt{998152685494675692076798951202529902815494250708576145481909686608666665739 \ldots}
    \end{align*}
    and the Lyapunov exponent is
    \begin{align*}
        \mathtt{0.}&\mathtt{1420636353704076392885014755358595810173327571468750083195247515173672923639} \\
                   &\mathtt{8074227810316135372232823286073279054769193620268515203117100954092496133766} \\
                   &\mathtt{1109580785592485411024630469429367517715069585739200674317420763635804075582} \\
                   &\mathtt{2632113670748717959061808939128761984643293098948941247048418492010098784074 \ldots}
    \end{align*}
    Such a computation takes a couple of days on a laptop.
\end{example}

\begin{example}
    (Section~\ref{sec:cond})
    For all $\gamma \in [1 - 10^{-51}, 1 + 10^{-51}]$,
    \[
        \Bigl(\frac{h'}{h}\Bigr)' \geq 0.57163
        \quad \text{and} \quad
        \Bigl(\frac{h''}{h}\Bigr)' \leq -1.91591
        \quad \text{on} \quad (0,1]
        .
    \]
    This computation is low precision; it takes seconds on a laptop, or minutes in a web browser.
\end{example}

Our code and a web demo are publicly available on github:\\
\hspace*{1cm} \url{https://github.com/khumarahn/khumarahn.github.io/tree/main/KWW26} \\
\hspace*{1cm} \url{https://khumarahn.github.io/KWW26}

\section{Abel function}
\label{sec:Abel}

Let $f \colon x \mapsto x (1 + 2^\gamma x^\gamma)$ be the left branch of~\eqref{eq:LSV}.
In this section we construct an Abel function $A(x)$, so that
\begin{equation}
    \label{eq:abelA}
    A(x) = A(f(x)) + 1
    .
\end{equation}
Our treatment is explicit and specific to our transformation;
see e.g.\ Szekeres~\cite{S58} for thorough and general theory.

\begin{figure}[h!]\label{fig:A}
    \centering
    \begin{tikzpicture}
        \begin{axis}[
            height       = 2in,
            xmin         = 0, xmax = 1.1,
            ymin         = 0, ymax = 25,
            xtick        = {0.0, 0.5, 1.0},
            xticklabels  = {$0$, $1/2$, $1$},
            ytick        = {0, 10, 20, 30},
            axis x line  = bottom,
            axis y line  = left,
            line cap     = round,
            ]
            \addplot[red,thick] coordinates {
                    (0.02,25.1827) (0.03,17.6356) (0.04,13.7201) (0.05,11.2967)
                    (0.06,9.63602) (0.07,8.41981) (0.08,7.48632) (0.10,6.13859)
                    (0.12,5.20428) (0.15,4.22898) (0.18,3.54689) (0.20,3.19297)
                    (0.25,2.52690) (0.30,2.05508) (0.35,1.69922) (0.40,1.41878)
                    (0.45,1.19050) (0.50,1.00000) (0.55,0.83787) (0.60,0.69767)
                    (0.65,0.57483) (0.70,0.46600) (0.80,0.28094) (0.90,0.12852)
                    (1.00,0)
                };
            \node[red,anchor=north west] at (axis cs:0.03,24) {$A(x)$};
        \end{axis}
    \end{tikzpicture}
    \caption{Abel function $A(x)$ for $\gamma = 95/100$.}
\end{figure}
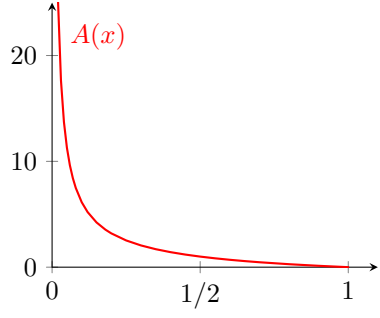

We prove that $A(x)$ is analytic in a complex region near $0$, and we
learn to compute it accurately and efficiently. The main result is
Theorem~\ref{thm:A}.

It is convenient to work with $f(x)$ using the variable $z = x^\gamma$ and the
conjugate map $\hf \colon z \mapsto z(1 + 2^\gamma z)^\gamma$.
Note that $\hf$ is analytic and invertible in a neighborhood of $z = 0$.

Let $r = 2^\gamma (\gamma + 1)$ and consider the {\it petal}
\[
    \cP := \{z : \Re (z^{-1}) \geq r\} \subset \C
    .
\]
Fix $n \geq 1$ and consider the truncated Abel function
\begin{equation}
    \label{eq:An}
    \hA_n(z) = \frac{a_{-1}}{z} + a_\ell \log z + a_0 + a_1 z + \cdots + a_n z^n
    ,
\end{equation}
where the coefficients $a_{-1}, a_\ell, a_1, a_2, \ldots$ are chosen uniquely so that
$|\hA_n(z) - \hA_n(\hf(z)) - 1| = O(z^{n + 2})$ as $z \to 0$, see~\eqref{eq:an},
and $a_0$ is now arbitrary and later chosen so that $\hA_n(\hf^{-k}(1)) \approx k$
for large $k$.

We construct the {\it principal} Abel function in $z$-coordinates by
\begin{equation}
    \label{eq:A}
    \hA(z)
    = \lim_{k \to \infty} \bigl[\hA_n(\hf^{-k}(z)) - k \bigr]
    .
\end{equation}
This definition makes sense, and $\hA$ is analytic on $\cP$,
see Lemma~\ref{lem:A} below.

From the definition of $\hA$, it is immediate that
\begin{equation}
    \label{eq:abel}
    \hA(\hf^{-1}(z)) = \hA(z) + 1
    , \quad
    \hA(z) = \hA(\hf(z)) + 1
    .
\end{equation}
Define $A(x) = \hA(x^\gamma)$; note that~\eqref{eq:abelA} holds, and $A$ is analytic on
$\cP^{1 / \gamma}$.

Let
\[
    D_n(z)
    = \hA_n(\hf(z)) - \hA_n(z) + 1
    \qquad \text{and} \qquad
    M
    = \sup_{z \in \cP} \frac{|D_n(z)|}{|z|^{n+1}}
    .
\]
Denote
\[
    I_n
    = \int_0^\infty \frac{1}{(1 + s^2)^{(n+1)/2}} \, ds
    = \frac{\sqrt{\pi}}{2} \frac{\Gamma(\frac{n}{2})}{\Gamma(\frac{n+1}{2})}
    .
\]

\begin{thm}
    \label{thm:A}
    For all $z \in \cP$,
    \[
        |\hA(z) - \hA_n(z)|
        \leq C_0 |z|^n
        ,
        \qquad \text{where} \qquad
        C_0 = \frac{ 2 M }{\gamma 2^\gamma} I_n
        .
    \]
\end{thm}

In the rest of this section we justify definition~\eqref{eq:A}
and prove Theorem~\ref{thm:A}.
The constant $M$ is estimated later in Lemma~\ref{lem:M}.

First, we make explicit the coefficients $a_{-1}, a_\ell, a_1, \ldots, a_n$. They are obtained by
matching the coefficients at $1, z, \ldots, z^{n+1}$ in the Taylor expansion of~\eqref{eq:abel}:
\begin{equation}
    \label{eq:an}
    \begin{pmatrix}
        -\gamma b & 0 & 0 & \cdots & 0 \\
        b^2 \binom{-\gamma}{2} & -\gamma b & 0 & \cdots & 0 \\
        b^3 \binom{-\gamma}{3} & \frac{\gamma b^2}{2} & \gamma b & \cdots & 0 \\
        \vdots & \vdots & \vdots & \ddots & \vdots \\
        b^{n+2} \binom{-\gamma}{n+2} & -\gamma\frac{(-1)^{n} b^{n+1}}{n+1} & b^{n} \binom{\gamma}{n} & \cdots & n\gamma b
    \end{pmatrix}
    \begin{pmatrix}
        a_{-1} \\
        -a_\ell \\
        a_1 \\
        \vdots \\
        a_n
    \end{pmatrix}
    =
    \begin{pmatrix}
        -1 \\
        0 \\
        0 \\
        \vdots \\
        0
    \end{pmatrix}
    .
\end{equation}
In particular,
\[
    a_{-1} = \frac{1}{\gamma 2^\gamma}
    , \quad \text{and} \quad
    a_\ell = - \frac{1}{2} - \frac{1}{2 \gamma}
    .
\]

Let us consider the (second) change of coordinates
\[
    t = 1/z
    .
\]
Denote $z_k = \hf^{-k}(z)$, and $t_k = 1/z_k$.
Throughout, set
\[
    b = 2^\gamma
    , \qquad
    c = \gamma 2^\gamma
    .
\]

\begin{lemma}
    \label{lem:tk}
    Let $R(t) = t - t(1+b/t)^{-\gamma} - c$,
    so that $t_{j+1} - t_j = c + R(t_{j+1})$.
    Then
    \[
        |R(t)|
        \leq \frac{b^2 \gamma (\gamma + 1)}{2 |t|}
        ,
    \]
    and in particular if $\Re(t) \geq r$, then
    $\Re(t_j) - \Re(t) \geq c j / 2$.
\end{lemma}

\begin{proof}
    We use
    \[
        g(x) = (1+x)^{-\gamma}
        = 1 - \gamma x + x^2 \int_0^1 (1 - s) g''(s x) \, ds
    \]
    with $g''(s x) = \gamma (\gamma + 1) (1 + s x)^{-\gamma-2}$.
    Thus
    \begin{align*}
        R(t)
            & = t - t(1+b/t)^{-\gamma} - c
            \\
            & = - t \Bigl( \frac{b}{t} \Bigr)^2 \gamma (\gamma + 1)
            \int_0^1 (1 - s) \Bigl(1 + s \frac{b}{t}\Bigr)^{-\gamma - 2} \, ds
            .
    \end{align*}
    Since $|1 + s b / t| > 1$ for $t$ with $\Re(t) > 0$, the result follows.
\end{proof}

\begin{lemma}
    \label{lem:A}
    $\hA$ is well defined by~\eqref{eq:A} (in particular is independent of $n$), and is analytic in $\cP$. The limit
    converges uniformly on compact subsets of $\cP$.
\end{lemma}

\begin{proof}
    Write
    \begin{equation}
        \label{lem:ADn}
        \hA_n (\hf^{-k}(z)) - k
        = \hA_n(z)
        - \sum_{j = 1}^k D_n(\hf^{-j}(z))
        .
    \end{equation}
    By Lemma~\ref{lem:tk}, $\hf^{-1}(\cP) \subset \cP$ and
    $|\hf^{-j}(z)| \lesssim j^{-1}$ uniformly on compact subsets of $\cP$.
    By the choice of coefficients of $\hA_n$, the terms on the right hand side above are
    bounded, up to a multiplicative constant, by $|\hf^{-j}(z)|^{n+1} \lesssim j^{-(n + 1)}$.
    Hence $\hA$ is defined as a limit of analytic functions on $\cP$ converging uniformly
    on compact subsets of $\cP$. Furthermore, $\hat A_m(z) - \hat A_n(z) \to 0$ as $z \to 0$, so the limit is independent of $n$. The result follows.
\end{proof}

Denote
\begin{equation}
    \label{eq:En}
    E_n(z)
    = \hA(z) - \hA_n(z)
    .
\end{equation}

\begin{cor}
    \label{cor:E}
    $E_n(z) = - \sum_{j=1}^\infty D_n(\hf^{-j}(z))$.
\end{cor}

\begin{proof}[Proof of Theorem~\ref{thm:A}]
    Note that $t_k = t_{k+1} (1 + b / t_{k+1})^{-\gamma}$. By Lemma~\ref{lem:Imt},
    if $\Re(t) \geq b (\gamma + 1)$, then for all $k \geq 0$,
    \[
        |\Im(t_{k+1})| \geq |\Im(t_k)|
        .
    \]
    With Lemma~\ref{lem:tk}, this implies
    \[
        |t_j|^2
        \geq |t|^2 + (c j / 2)^2
        .
    \]
    Using Corollary~\ref{cor:E},
    \begin{align*}
        |E_n(z)|
            & \leq \sum_{j=1}^\infty |D_n(\hf^{-j}(z))|
            \leq M \sum_{j=1}^\infty |\hf^{-j}(z)|^{n+1}
            = M \sum_{j=1}^\infty \frac{1}{|t_j|^{n+1}}
            \\
            & \leq M \sum_{j=1}^\infty
            \frac{1}{(|t|^2 + (c j / 2)^2)^{(n + 1) / 2}}
            \leq M \int_0^\infty
            \frac{1}{(|t|^2 + (c u / 2)^2)^{(n + 1) / 2}} \, du
            \\
            & = \frac{ 2 M }{c |t|^n} \int_0^\infty
            \frac{1}{(1 + s^2)^{(n + 1) / 2}} \, ds
            = \frac{ 2 M }{c |t|^n} I_n
            .
    \end{align*}
\end{proof}

\subsection{Estimation of \texorpdfstring{$M$}{M}}
\label{sec:M}

Now we estimate $M$. Recall that $D_n(z) = \hA_n(\hf(z)) - \hA_n(z) + 1$,
so for $|z| < 1 / b$,
\begin{align*}
    D_n(z)
    & = 1 + \frac{a_{-1}}{z} \bigl[ (1 + b z)^{-\gamma} - 1 \bigr]
    + a_\ell \gamma \log(1 + b z)
    + \sum_{k = 1}^n a_k z^k \bigl[ (1 + b z)^{k \gamma} - 1 \bigr]
    = \sum_{j = 0}^\infty c_j z^j
    ,
\end{align*}
where $c_0 = 1 - \gamma b a_{-1}$ and for $j \geq 1$,
\[
    c_j  = a_{-1} b^{j+1} \binom{-\gamma}{j+1}
    + a_\ell \gamma \frac{(-1)^{j-1} b^j}{j}
    + \sum_{k=1}^{\min(n, j-1)} a_k b^{j-k} \binom{k\gamma}{j-k}
    .
\]

\begin{rmk}
    \label{rmk:cj}
    By~\eqref{eq:an}, $c_j = 0$ for $j \leq n$.
\end{rmk}

\begin{lemma}
    \label{lem:M}
    Let
    \[
        G(z)
        = \frac{|a_{-1}|}{z} \bigl[ (1 - b z)^{-\gamma} - 1 \bigr]
        - |a_\ell| \gamma \log(1 - b z)
        + \sum_{k = 1}^n |a_k| z^k \bigl[ (1 - b z)^{-k \gamma} - 1 \bigr]
        .
    \]
    Let $g_j$ be the Maclaurin coefficients of $G(z)$,
    given explicitly by $g_0 = |a_{-1}| b \gamma$, and for $j \geq 1$,
    \begin{align*}
        g_j & = |a_{-1}| b^{j+1} \binom{j+\gamma}{j+1}
        + |a_\ell| \gamma \frac{b^j}{j}
        + \sum_{k=1}^{\min(n, j-1)} |a_k| b^{j-k} \binom{j-k+k\gamma-1}{j-k}
        .
    \end{align*}
    Then for each $K \geq 0$,
    \[
        \sup_{z \in \cP} \frac{|D_n(z)|}{|z|^{n+1}}
        \leq M
        \qquad \text{with} \qquad
        M
        = r^{n + 1} \biggl[
            \sum_{j = n + 1}^{n + K} |c_j| r^{-j}
            + G(r^{-1}) -\sum_{j = 0}^{n + K } g_j r^{-j}
        \biggr]
        .
    \]
\end{lemma}

\begin{proof}
    Note that $G(z)$ and $D_n(z)$
    are analytic for $|z| < b^{-1}$, and
    by Remark~\ref{rmk:cj}, $c_j = 0$ for $j \leq n$.
    By construction, $g_j \geq |c_j|$ for all $j$.
    Thus, for any $z$ with $|z| < b^{-1}$:
    \begin{align*}
        \frac{|D_n(z)|}{|z|^{n + 1}}
        & \leq \frac{1}{|z|^{n + 1}} \biggl[
            \Bigl| \sum_{j = n + 1}^{n + K} c_j z^j \Bigr|
            + \sum_{j = n + K + 1}^\infty |c_j| |z|^j
        \biggr]
        \\
        & \leq \frac{1}{|z|^{n + 1}} \biggl[
            \Bigl| \sum_{j = n + 1}^{n + K} c_j z^j \Bigr|
            + \sum_{j = n + K + 1}^\infty g_j |z|^j
        \biggr]
        \\
        & \leq \frac{1}{|z|^{n + 1}} \biggl[
            \sum_{j = n + 1}^{n + K} |c_j| \, |z|^j
        + G(|z|) -\sum_{j = 0}^{n + K } g_j |z|^j
        \biggr]
        .
    \end{align*}
    The right hand side above has all positive coefficients and
    is therefore increasing with $|z|$.
    Since $z \in \cP$ implies $|z| \leq r^{-1} \leq b^{-1}$,
    the result follows.
\end{proof}

\subsection{\texorpdfstring{$t$}{t}-domain and derivatives}

It will be convenient for us to use the variable $t = z^{-1}$ and consider
$\tA(t) = \hA(t^{-1})$ and $\tA_n(t) = \hA_n(t^{-1})$. Then
\begin{equation}
    \label{eq:Atn}
    \tA_n(t) = a_{-1} t - a_\ell \log t + a_0 + \frac{a_1}{t} + \cdots + \frac{a_n}{t^n}
    .
\end{equation}

Let $\tcP = \cP^{-1} = \{ t \in \C : \Re(t) \geq r \}$ be the petal $\cP$ in the $t$-variable.

Let $\cP_1$, $\tcP_1$ denote a slightly smaller petal than $\cP$:
\begin{equation}
    \label{eq:cP1}
    \cP_1 = \{ z \in \C : \Re(z^{-1}) \geq r_1 \}
    , \quad
    \tcP_1 = \{ t \in \C : \Re(t) \geq r_1 \}
    ,
\end{equation}
where $r_1 > r + 1$.
Let $C_0$ be as in Theorem~\ref{thm:A}.

\begin{lemma}
    \label{lem:CAn}
    $|\tA^{(k)}(t) - \tA_n^{(k)}(t)| \leq C_0 k! (|t| - 1)^{-n}$ for $k \geq 1$ and $t \in \tcP_1$.
    Subsequently,
    \[
        \sup_{t \in \tcP_1} |\tA^{(k)}(t) - \tA_n^{(k)}(t)|
        \leq C_0 k! (r_1 - 1)^{-n}
        .
    \]
\end{lemma}

\begin{proof}
    $\hat A_n$ is analytic on $\cP$ by construction, and $\hat A$ is analytic on $\cP$ by Lemma~\ref{lem:A}.

    We can therefore apply the Cauchy maximum principle and Theorem~\ref{thm:A} to get that
    \[
        |\tA^{(k)}(t) - \tA_n^{(k)}(t)|
        \leq k! \sup_{|s - t| = 1} |\tA(s) - \tA_n(s)|
        \leq k! \sup_{|s - t| = 1} C_0 |s|^{-n}
        = C_0 k! (|t| - 1)^{-n}
        .
    \]
\end{proof}

\begin{lemma}
    \label{lem:CA}
    For $t \in \tcP_1$,
    \[
        |t (\tA'(t) - a_{-1})|
        \leq r_1 C_1
        ,
    \]
    where
    \[
        C_1
        = \frac{|a_\ell|}{r_1} + \sum_{k = 1}^{n} \frac{k |a_k|}{r_1^{k+1}}
        + \frac{C_0}{(r_1 - 1)^n}
        .
    \]
    Consequently,
    \[
        \sup_{t \in \tcP_1} |\tA'(t) - a_{-1}|
        = \sup_{z \in \cP_1} |z^2 \hA'(z) + a_{-1}|
        \leq C_1
        ,
    \]
    and, assuming $C_1 < a_{-1}$,
    \[
        \sup_{w \in \tA(\tcP_1)} \bigl| (\tA^{-1})'(w) - a_{-1}' \bigr|
        \leq C_1'
        ,
    \]
    where
    \[
        a_{-1}' = \frac{a_{-1}}{a_{-1}^2 - C_1^2}
        \quad \text{and} \quad
        C_1' = \frac{C_1}{a_{-1}^2 - C_1^2}
        .
    \]
\end{lemma}

\begin{proof}
    Suppose that $t \in \tcP_1$.
    Using~\eqref{eq:Atn} and the triangle inequality,
    \[
        |t(\tA_n'(t) - a_{-1})|
        = \Bigl| a_\ell + \frac{a_1}{t} + \cdots + \frac{n a_n}{t^n} \Bigr|
        \leq |a_\ell| + \sum_{k = 1}^{n} \frac{k |a_k|}{r_1^k}
        .
    \]
    Write
    \[
        |t(\tA'(t) - a_{-1})|
        \leq |t(\tA_n'(t) - a_{-1})| + |t(\tA'(t) - \tA_n'(t))|
        .
    \]
    By Lemma~\ref{lem:CAn}, $|\tA'(t) - \tA_n'(t)| \leq C_0 (|t| - 1)^{-n}$.
    Since $|t| (|t| - 1)^{-n}$ is monotonically decreasing in $|t|$ for $|t| > 1$ and $n \geq 1$,
    its supremum on $\tcP_1$ is attained at $|t| = r_1$. Combining these bounds, we get
    $|t(\tA'(t) - a_{-1})| \leq r_1 C_1$. The remaining bounds are straightforward.
\end{proof}

\begin{rmk}
    \label{rmk:a1_m_C1}
    Everywhere further we assume that $r_1$ is large enough so that $C_1 < a_{-1}$.
\end{rmk}

\begin{lemma}
    \label{lem:CA2}
    For $t \in \tcP_1$,
    \[
        |t^2 \tA''(t) - a_\ell|
        \leq \sum_{k=1}^n \frac{k (k+1) |a_k|}{|t|^k} + \frac{2 C_0 |t|^2}{(|t| - 1)^n}
        .
    \]
    Subsequently,
    \[
        \sup_{z \in \cP_1} |z^3 \hA''(z) - 2 a_{-1}|
        \leq C_2
        ,
        \quad \text{where} \quad
        C_2
    = 2 C_1 + \frac{|a_\ell|}{r_1}
        + \sum_{k=1}^n \frac{k (k+1) |a_k|}{r_1^{k+1}} + \frac{2 C_0 r_1}{(r_1 - 1)^n}
        .
    \]
\end{lemma}

\begin{proof}
    The proof is similar to that of Lemma~\ref{lem:CA}. We use
    $z^3 \hA''(z) = t \tA''(t) + 2 \tA'(t)$.
\end{proof}

We defined $\hA$ on the petal $\cP$. Naturally, the relation
$\hA(z) = \hA(\hf^{-n}(z)) - n$ extends $\hA$ to all $z$ that enter $\cP$ under the
iteration of $\hf^{-1}$. Computing integrals in the Euler-Maclaurin formula, we will
need the following lemma:

\begin{lemma}
    \label{lem:Ap}
    $\hA' < 0$ on $(0,\infty)$.
\end{lemma}

\begin{proof}
    By Lemma~\ref{lem:CA} and Remark~\ref{rmk:a1_m_C1}, $\hA'(z) < 0$ on $\cP_1 \cap (0, \infty)$.
    Differentiating $\hA(\hf(w)) = \hA(w) - 1$ yields $\hA'(\hf(w))\hf'(w) = \hA'(w)$.
    Since $\hf'$ is positive on $(0, \infty)$, the negative sign of $\hA'$ is preserved under
    iterations of $\hf$. The result follows.
\end{proof}

\section{Sums for induced transfer operator}
\label{sec:sums}

In this section, we approximate sums of the type
\begin{equation}
    \label{eq:S}
    S(x) = \sum_{k \geq 0} \frac{\varphi(x, k)}{J_k(x)}
    ,
\end{equation}
where $\varphi$ is a sufficiently nice function and
\[
    J_k(x)
    = (f^k)' (f^{-k}(x))
    = \prod_{j = 1}^k (1 + (\gamma + 1) 2^\gamma x_j^\gamma)
    .
\]

These sums appear in the expression for the transfer operator $\cL$
of the first return map to $[1/2,1]$ with respect to Lebesgue measure:
\begin{equation}
    \label{eq:L}
    (\cL v)(x)
    = \sum_{k \geq 0} \frac{v \bigl( \frac{x_k + 1}{2} \bigr)}{2 J_k(x)}
    , \quad \text{where} \quad
    x_k = f^{-k}(x)
    .
\end{equation}

\begin{rmk}
    \label{rmk:Sz}
    Note that
    \[
        S(x) = \sum_{k=0}^N \frac{\varphi(x, k)}{J_k(x)}
        + \frac{1}{J_{N+1}(x)} \sum_{k=0}^\infty \frac{\varphi(x, N + 1 + k)}{J_k(x_{N+1})}
        .
    \]
    This is useful in reducing the computation of $S(x)$ for large $x$ to that with small $x$.
\end{rmk}

\subsection{Euler-Maclaurin formula}

We use the following abstract result which, with a minor modification, we take
from~\cite[Section~5]{VW25}:

\begin{thm}
    \label{thm:EM}
    Consider the sum $\sum_{k = 0}^\infty \psi(k)$, where for each real $s \geq 0$,
    the function $z \mapsto \psi(s + z)$
    is analytic with absolute value bounded by $C_\psi$ in the disk centered at $s$
    of radius $\varkappa_1 (\nu + s)$, where $C_\psi > 0$, $\nu > 0$ and $\varkappa_1 \in (0,1]$
    are constants independent of $s$.
    Suppose that $L \geq 1$.
    Then
    \[
        \sum_{k = 0}^\infty \psi(k)
        = \int_0^\infty \psi(s) \, ds
        + \frac{1}{2} \psi(0)
        - D_L
        + R_L
        ,
    \]
    where
    \[
        D_L
        = \sum_{j=1}^{L} \frac{B_{2j}}{(2j)!} \psi^{(2j-1)}(0)
        \quad \text{and} \quad
        |R_L|
        \leq \frac{(2L+1)! C_\psi}{L (2\pi)^{2L+1} \nu^{2L} \varkappa_1^{2 L + 1}}
        .
    \]
    Further, if $L < \pi e \nu \varkappa_1$ and $M \geq 2 L$, then
    \[
        \Bigl| D_L - \frac{1}{M} \sum_{m=1}^M c_m \psi(s_{m}) \Bigr|
        \leq \frac{\pi^2 e^2}{3}
        \cdot \frac{1}{\sqrt{2L - 1}}
        \cdot \frac{
            \nu \varkappa_1
        }{
            \bigl(\frac{2 \pi e \nu \varkappa_1}{2L-1}\bigr)^2 - 1
        }
        \cdot \frac{C_\psi}{e^M - 1}
        ,
    \]
    where $\tau = \varkappa_1 \nu / e$ and
    \[
        s_m = \tau e^{2 \pi i \frac{2 m - 1}{2 M}}
        \quad \text{and} \quad
        c_m = \sum_{\ell = 1}^L \frac{B_{2 \ell}}{2 \ell} \tau^{-(2 \ell-1)}
        e^{- 2 \pi i \frac{(2 \ell-1)(2 m - 1)}{2 M}}
        , \quad m = 1, \ldots , M
        .
    \]
\end{thm}

\begin{rmk}
    If $\psi(\overline{z}) = \overline{\psi(z)}$, and $M$ is even, then
    \[
        \frac{1}{M} \sum_{m=1}^M c_m \psi(s_{m})
        = \frac{2}{M} \sum_{m=1}^{M / 2} \Re(c_m \psi(s_{m}))
        .
    \]
    This is useful to speed up the computation.
\end{rmk}

\begin{rmk}
    We note a few errors in~\cite{VW25} related to the derivative term $D_L$:
    \begin{enumerate}
        \item Wrong sign at $D_L$ in the Euler-Maclaurin
            formula~\cite[(33)]{VW25};
            though the sign is correct in their code.
        \item In the proof of~\cite[Proposition~5.4]{VW25},
            the factor $-1$ is missing in $\frac{B_{2 l}}{(2 l)!} = \frac{(-1)^l 2 \zeta(2 l)}{(2 \pi)^{2 l}}$,
            and the $2$ at $\zeta(\cdot)$ is later lost.
        \item In the same proof, the bound $(2 l - 1)! \leq (2 l - 1)^{2 l} e^{-(2 l - 1)}$ fails for $l \leq 3$.
    \end{enumerate}
    The bound on $D_L$ in Theorem~\ref{thm:EM} is obtained using a version of Robbins bound
    $n! \leq n^{n + 1/2} e^{1 - n}$, fixing the sign and the missing factor of $2$,
    and otherwise closely following the original proof.
\end{rmk}

\subsection{Preliminary bounds}

Now we apply Theorem~\ref{thm:EM} to the sums $S(x)$
when $x \in \cP_1^{1 / \gamma}$. We fix $x$, let $z = x^\gamma$, $w = A(x)$,
and write
\[
    S(x)
    = \sum_{k = 0}^\infty \frac{\varphi(x, k)}{J_k(x)}
    = \sum_{k = 0}^\infty \psi(k)
    ,
\]
where for complex $s$,
\begin{equation}
    \label{eq:psi}
    \psi(s)
    = \frac{\varphi(A^{-1}(w), s)}{J_s(x)}
    , \quad
    J_s (x) = \frac{\kappa(w + s)}{\kappa(w)}
    \quad \text{and} \quad
    \kappa(s)
    = \frac{1}{\gamma} A'(A^{-1}(s))
    .
\end{equation}
Assume that $\psi$ is analytic in each disk of radius $\varkappa_1(\nu + s)$ around $s$ for each $s \geq 0$.

Choose $\nu$ and $\varkappa_1$ so that
\begin{equation}
    \label{eq:nuka}
    \nu \leq (\Re(z^{-1}) - r_1)(a_{-1} + C_1)
    \quad \text{and} \quad
    \varkappa_1
    \leq \varkappa_0
    , \quad
    \varkappa_0
    = \frac{a_{-1} - C_1}{a_{-1} + C_1}
    .
\end{equation}

To control the ratio $\frac{\kappa(w)}{\kappa(w + \cdot)}$ in these disks, we use the following lemma.

\begin{lemma}
    \label{lem:disk_inclusion}
    For all $s \geq 0$ and $|\zeta| \leq \varkappa_1(\nu + s)$,
    the point $w + s + \zeta$ is contained in $\hA(\cP_1)$.
\end{lemma}

\begin{proof}
    Let $t = z^{-1} \in \tcP_1$, so that $w = \tA(t)$.
    Let $t' = \tA^{-1}(w + s)$ and $t'' = \tA^{-1}(w + s + \zeta)$.
    We must show that $\Re(t'') \geq r_1$.

    First, consider the horizontal path
    $P_1(u) = w + u s$ parametrized by $u \in [0, 1]$.
    By~\eqref{eq:nuka}, we have $\Re(t) \geq r_1 + \frac{\nu}{a_{-1} + C_1} \geq r_1$.
    Let
    \[
        u^*
        = \sup \bigl\{
            \tau \in [0,1] : \Re(\tA^{-1}(P_1(u))) \geq r_1 \text{ for all } u \in [0, \tau]
        \bigr\}
        .
    \]
    Then $P_1(u) \in \tA(\tcP_1)$ for $u \leq u^*$.
    By Lemma~\ref{lem:CA}, $\Re\bigl((\tA^{-1})'\bigr) \geq \frac{1}{a_{-1} + C_1}$ on $\tA(\tcP_1)$.
    Then
    \[
        \Re \bigl( \tA^{-1}(P_1(u)) \bigr)
        \geq \Re(t) + \int_0^u \frac{s}{a_{-1} + C_1} \, d\xi
        = \Re(t) + \frac{u s}{a_{-1} + C_1}
        .
    \]
    Since this is strictly larger than $r_1$ for all $u \leq u^*$, we must have $u^* = 1$.
    Evaluating at $u = 1$ gives $\Re(t') \geq \Re(t) + \frac{s}{a_{-1} + C_1}$.

    Next we apply a similar treatment to the straight-line path
    $P_2(v) = w + s + v\zeta$, parametrized by $v \in [0, 1]$.
    Let
    \[
        v^*
        = \sup \bigl\{
            \tau \in [0,1] : \Re(\tA^{-1}(P_2(v))) \geq r_1 \text{ for all } v \in [0, \tau]
        \bigr\}
        .
    \]
    For $v < v^*$, $P_2(v) \in \tA(\tcP_1)$, and by Lemma~\ref{lem:CA},
    $|(\tA^{-1})'| \leq (a_{-1} - C_1)^{-1}$ on $\tA(\tcP_1)$.
    Consequently,
    \[
        \Re \bigl( \tA^{-1}(P_2(v)) \bigr)
        \geq \Re(t') - \frac{v |\zeta|}{a_{-1} - C_1}
        \geq \Re(t) + \frac{s}{a_{-1} + C_1} - \frac{v \varkappa_1(\nu + s)}{a_{-1} - C_1}
        .
    \]
    Recall that $\Re(t) \geq r_1 + \frac{\nu}{a_{-1} + C_1}$. Then for all $v < v^*$,
    \[
        \Re \bigl( \tA^{-1}(P_2(v)) \bigr)
        \geq r_1 + (\nu + s) \Bigl[
            \frac{1}{a_{-1} + C_1} - \frac{v \varkappa_1}{a_{-1} - C_1}
        \Bigr]
        .
    \]
    By the definition of $\varkappa_1$, the term in the brackets is positive as long as $v < 1$.
    This ensures $v^* = 1$. Thus $\Re(t'') \geq r_1$, completing the proof.
\end{proof}

\begin{figure}[htbp]
    \centering
    \begin{tikzpicture}[
        scale=0.95,
        every node/.style={text=black, font=\small},
        dot/.style={circle, fill, inner sep=1.2pt}
    ]
        \colorlet{pair1}{red!80!black}     
        \colorlet{pair2}{blue!80!black}    
        \colorlet{pair3}{green!60!black}   
        \colorlet{pair4}{orange!90!black}  

        \begin{scope}[shift={(0,0)}]
            \draw[->, thick] (-1,0) -- (4.5,0) node[right] {$\Re$};
            \draw[->, thick] (0,-1.5) -- (0,3.5) node[above] {$\Im$};

            \coordinate (Ar1) at (0.6, 0);
            \draw[dashed] plot[domain=-1.5:3.2, smooth, samples=50, variable=\y] ({0.6 + 0.35*ln(1 + abs(\y))}, {\y});
            \node[dot, fill=pair1, label={[text=pair1, font=\small]below right:$\tilde{A}(r_1)$}] at (Ar1) {};

            \coordinate (apex) at (1.2, 1.4);
            \coordinate (w) at (2.0, 1.4);
            \coordinate (ws) at (3.5, 1.4);
            \coordinate (zeta_pt) at ($(ws) + (50:1.2)$);



            \draw[->, thick] (w) -- (ws) node[midway, below] {$s$};
            \draw[->, thick] (ws) -- (zeta_pt) node[midway, right=1pt] {$\zeta$};

            \draw[dashed, thick] (ws) circle (1.2);

            \node[dot, fill=pair2, label={[text=pair2, font=\small]below left:$w$}] at (w) {};
            \node[dot, fill=pair3, label={[text=pair3, font=\small]below right:$w+s$}] at (ws) {};
            \node[dot, fill=pair4, label={[text=pair4, font=\small]right:$w+s+\zeta$}] at (zeta_pt) {};

            \node[font=\normalsize] at (1.75, -2.2) {Abel-plane};
        \end{scope}

        \draw[->, thick, -stealth] (5.2, 1.4) -- (6.5, 1.4) node[midway, above, font=\small] {$\tilde{A}^{-1}$};

        \begin{scope}[shift={(7.8,0)}]
            \draw[->, thick] (-1,0) -- (4.5,0) node[right] {$\Re$};
            \draw[->, thick] (0,-1.5) -- (0,3.5) node[above] {$\Im$};

            \coordinate (r1) at (0.8, 0);
            \draw[dashed] (r1) -- (0.8, 3.2);
            \draw[dashed] (0.8, -1.5) -- (r1);
            \node[dot, fill=pair1, label={[text=pair1, font=\small]below right:$r_1$}] at (r1) {};

            \coordinate (t) at (1.6, 0.8);
            \coordinate (t_ws) at (3.2, 1.6);
            \coordinate (t_zeta) at ($(t_ws) + (40:1.3)$);

            \draw[thick] (t) -- (t_ws);
            \draw[thick] (t_ws) -- (t_zeta);

            \node[dot, fill=pair2, label={[text=pair2, font=\small]below left:$t$}] at (t) {};
            \node[dot, fill=pair3, label={[text=pair3, font=\small]below right:$\tilde{A}^{-1}(w+s)$}] at (t_ws) {};
            \node[dot, fill=pair4, label={[text=pair4, font=\small]above:{$t' = \tilde{A}^{-1}(w+s+\zeta)$}}] at (t_zeta) {};

            \node[font=\normalsize] at (1.75, -2.2) {$t$-plane};
        \end{scope}

    \end{tikzpicture}
    \caption{Mapping of the integration domain via $\tilde{A}^{-1}$.
    The dashed curve represents the boundary of $\tA(\tcP_1)$.
    The components mapping onto the $t$-plane illustrate the progressive boundary
    constraints established in Lemma \ref{lem:disk_inclusion}.
    }
    \label{fig:disk_inclusion}
\end{figure}
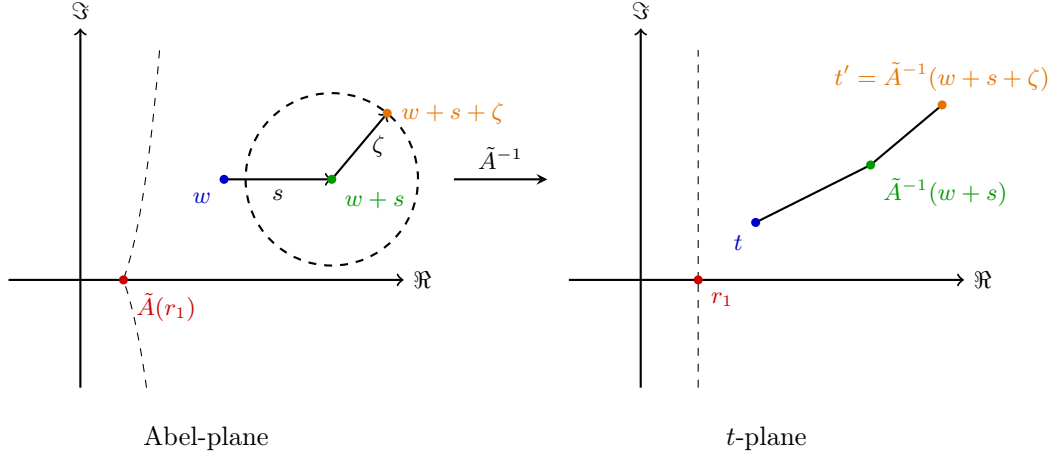

Note that $\psi(0) = \varphi(x, 0)$, and with the change of variables
$u = A^{-1}(w + s)$, the integral term simplifies:
\begin{equation}
    \label{eq:EMint}
    \int_0^\infty \psi(s) \, ds
    = - A'(x) \int_0^x \varphi(x, A(u) - A(x)) \, du
    .
\end{equation}

To apply Theorem~\ref{thm:EM}, it remains to estimate $C_\psi$.

\begin{lemma}
    \label{lem:kk}
    Let $t = z^{-1}$.
    For all $s \geq 0$ and $|\zeta| \leq \varkappa_1(\nu + s)$,
    \begin{equation}
        \label{eq:kk}
        \Bigl| \frac{\kappa(w)}{\kappa(w+s+\zeta)} \Bigr|
        \leq \frac{a_{-1} + C_1}{a_{-1} - C_1} \,
        \Bigl[
            \frac{|t|}{|t + a_{-1}' s| - K s - C}
        \Bigr]^{1 + 1 / \gamma}
        ,
    \end{equation}
    where $C = \frac{\varkappa_1 \nu}{a_{-1} - C_1}$ and
    $K = C_1' + \frac{\varkappa_1}{a_{-1} - C_1}$.
\end{lemma}

\begin{proof}
    Write $\kappa(w) = -\tA'(t) t^{1 + 1/\gamma}$. With $t' = \tA^{-1}(w + s + \zeta)$,
    using Lemma~\ref{lem:CA},
    \[
        \Bigl| \frac{\kappa(w)}{\kappa(w+s+\zeta)} \Bigr|
        = \biggl| \frac{\tA'(t)}{\tA'(t')} \biggr|
        \, \Bigl| \frac{t}{t'} \Bigr|^{1 + 1 / \gamma}
        \leq \frac{a_{-1} + C_1}{a_{-1} - C_1}
        \, \Bigl| \frac{t}{t'} \Bigr|^{1 + 1 / \gamma}
        ,
    \]
    Also by Lemma~\ref{lem:CA}, $|t'| \ge |t + a_{-1}' s| - K s - C$,
    and the result follows.
\end{proof}

\subsection{Bounded observables}

In this section we prepare an application of Theorem~\ref{thm:EM}
when $\varphi$ is bounded, such as when $\varphi(x, k)$ is a
Chebyshev polynomial in $x_k = f^{-k}(x)$.
We assume that
\[
    \sup \{ |\varphi(A^{-1}(w), s + \zeta)| : s \geq 0, |\zeta|
    \leq \varkappa_1(s + \nu) \} \leq \eps
\]
with some $\eps \geq 0$.
Then we can choose any
\begin{align*}
    C_\psi
    \geq \eps \sup_{\substack{s \geq 0 \\ |\zeta| \leq \varkappa_1(\nu + s)}}
    \Bigl|
        \frac{\kappa(w)}{\kappa(w+s+\zeta)}
    \Bigr|
    ,
\end{align*}
and the bound on the first factor is given by the following lemma.

\begin{lemma}
    \label{lem:kappa}
    Let $t = z^{-1} = \tA^{-1}(w)$ and $C = \frac{\varkappa_1 \nu}{a_{-1} - C_1}$.
    For $\Re(t) > r_1+C$ and $|\zeta| \leq \varkappa_1(\nu + s)$,
    \[
        \Bigl| \frac{\kappa(w)}{\kappa(w+s+\zeta)} \Bigr|
        \leq \frac{a_{-1} + C_1}{a_{-1} - C_1}
        \max \Bigl\{
            \frac{\Re(t)}{\Re(t) - C},
            \frac{1}{R}
        \Bigr\}^{1 + 1/\gamma}
        ,
    \]
    where $R = \sqrt{1 - \bigl(q + \varkappa_1(1 + q)\bigr)^2}$
    with $q = C_1 / a_{-1}$. If $t$ is real, a simpler bound holds:
    \[
        \Bigl| \frac{\kappa(w)}{\kappa(w+s+\zeta)} \Bigr|
        \leq \frac{a_{-1} + C_1}{a_{-1} - C_1}
        \Bigl[ \frac{t}{t - C} \Bigr]^{1 + 1/\gamma}
        ,
    \]
\end{lemma}

\begin{proof}
    We apply Lemma~\ref{lem:kk}.
    Let $t = u + iv$. Then $|t + a_{-1}' s| = \sqrt{v^2 + (u + a_{-1}' s)^2}$ and
    \[
        \frac{|t'|}{|t|}
        \geq \frac{\sqrt{v^2 + (u + a_{-1}' s)^2} - K s - C}{\sqrt{u^2 + v^2}}
        .
    \]
    Let $\alpha \in (0, \pi/2)$ such that $\cos \alpha = K / a_{-1}'$ and
    $\sin \alpha = R$. We minimize over $s \ge 0$ for fixed $u, v$:
    \[
        \frac{|t'|}{|t|}
        \geq \begin{cases}
            1 - C / \sqrt{u^2 + v^2}
            , & |v| \leq u \tan \alpha , \\
            (|v| \sin \alpha + u \cos \alpha - C) / \sqrt{u^2 + v^2}
            , & |v| > u \tan \alpha .
        \end{cases}
    \]
    With $v = 0$, the result for real $t$ is immediate. For complex $t$ we
    minimize over $v$. For $|v| \leq u \tan \alpha$, the ratio is minimized at
    $v = 0$, giving $1 - C / u$. Otherwise, the ratio has no interior minimum,
    so its infimum occurs at the boundaries. With $|v| = u \tan \alpha$
    the value is $1 - (C \cos \alpha)/u$, and as $|v| \to \infty$,
    the limit is $\sin \alpha = R$. The desired result for complex $t$ follows.
\end{proof}

\begin{cor}
    \label{cor:kappa}
    In the setup of Lemma~\ref{lem:kappa}, choose
    \[
        r_*
        \geq \max \Bigl\{
            r_1 + \frac{\nu}{a_{-1}},
            \frac{C}{1 - R}
        \Bigr\}
        .
    \]
    Then for $\Re(z^{-1}) \geq r_*$,
    \[
        C_\psi
        \leq \frac{a_{-1} + C_1}{a_{-1} - C_1}
        R^{-(1 + 1/\gamma)}
        \sup_{z \in \cP_1} |\varphi(z)|
        .
    \]
\end{cor}

\begin{rmk}
    A natural choice of $\varkappa_1$ is
    \[
        \varkappa_1
        = \varkappa_0
        \frac{2 (a_{-1} r_1 + \nu) \nu}{\nu^2 + (a_{-1} r_1 + \nu)^2}
        ,
    \]
    designed to make the two terms in the expression for $r_*$ approximately
    equal when $C_1$ is small.
\end{rmk}

\begin{prop}
    \label{prop:sum-small}
    \[
        |S(z)|
        \leq \eps \gamma |\hA'(z) z| + \frac{\eps}{2}
        + C_\psi \Bigl[
            \sum_{j=1}^L \frac{|B_{2j}|}{2j (\nu \varkappa_1)^{2j-1}}
            + \frac{(2L+1)!}{L (2\pi)^{2L+1} \nu^{2L} \varkappa_1^{2L+1}}
        \Bigr]
        .
    \]
    Let $C, R$ be as in Lemma~\ref{lem:kappa}.
    Then for real $z$ with
    $z^{-1} > \max \bigl\{ r_1 + \frac{\nu}{a_{-1} + C_1}, C \bigr\}$,
    \[
        C_\psi
        \leq \eps \frac{a_{-1} + C_1}{a_{-1} - C_1}
        \Bigl[ \frac{z^{-1}}{z^{-1} - C} \Bigr]^{1 + 1/\gamma}
        .
    \]
    With $r_*$ as in Corollary~\ref{cor:kappa}, for complex $z$
    with $\Re(z^{-1}) \geq r_*$,
    \[
        C_\psi
        \leq \eps \frac{a_{-1} + C_1}{a_{-1} - C_1}
        R^{-(1 + 1/\gamma)}
        .
    \]
\end{prop}

\begin{proof}
    Recall the Euler-Maclaurin formula.
    First we bound the integral term. Since $|\varphi(u)| \leq \eps$,
    using equation~\eqref{eq:EMint} and $A'(x) = \gamma x^{\gamma-1} \hA'(z)$,
    \[
        \Bigl| - A'(x) \int_0^x \varphi(x, A(u) - A(x)) \, du \Bigr|
        \leq \eps \gamma |\hA'(z) z|
        .
    \]
    For the boundary term, $\psi(0) = \varphi(x, 0)$, so $\frac{1}{2} |\psi(0)| \leq \frac{\eps}{2}$.

    To bound the derivatives term $D_L = \sum_{j=1}^L \frac{B_{2j}}{(2j)!} \psi^{(2j-1)}(0)$,
    we apply Cauchy's estimate: $\psi(s)$ is bounded by $C_\psi$ and analytic in the disk
    of radius $\varkappa_1 \nu$ around $0$. Thus
    \[
        |\psi^{(2j-1)}(0)|
        \leq \frac{(2j-1)! C_\psi}{(\nu \varkappa_1)^{2j-1}}
        \quad \text{and} \quad
        |D_L|
        \leq \sum_{j=1}^L \frac{|B_{2j}|}{2j} \frac{C_\psi}{(\nu \varkappa_1)^{2j-1}}
        .
    \]
    Combining this with the bound for the remainder term $R_L$,
    and the bounds on $C_\psi$ by Lemma~\ref{lem:kappa}
    and Corollary~\ref{cor:kappa}, completes the proof.
\end{proof}

\subsection{Particular observables}
\label{sec:observables}

In this section we apply Theorem~\ref{thm:EM} to a few particular
choices of $\varphi(x, k)$. We show how to bound $C_\psi$ and
compute the integral $\int_0^\infty \psi(s) \, ds$.

\subsubsection{Chebyshev polynomials}
\label{sec:comp-cheb}

Let $\varphi(x, k) = \varphi_m(x_k)$ with $x_k = f^{-k}(x)$ and
\begin{equation}
    \label{eq:phi-n-cheb}
    \varphi_m (x) = T_m(2 x - 1)
    ,
\end{equation}
where $T_m$ is the $m$-th Chebyshev polynomial.
In the complex plane,
\[
    \varphi(x, s)
    = \varphi_m (A^{-1}(A(x) + s))
    .
\]
Denote $\|\varphi_m\|_{\cP_1} = \max_{z \in \cP_1} |\varphi_m(z^{1 / \gamma})|$.

In the code, we use the following observations:

\begin{rmk}
    It is standard (see e.g.~\cite[Section~1.4.2]{MH03}) that
    $\max_{|x| \leq \delta} |\varphi_m(x)| = | \varphi_m (-\delta) |$.
    Accordingly,  with $q = 2 r_1^{ - 1 / \gamma}$,
    \[
        \|\varphi_m\|_{\cP_1}
        = | \varphi_m (-r_1^{- 1 / \gamma}) |
        = \frac{1}{2} \Bigl[
            \Bigl( 1 + q + \sqrt{q^2 + 2 q} \Bigr)^m +
            \Bigl( 1 + q - \sqrt{q^2 + 2 q} \Bigr)^m
        \Bigr]
        .
    \]
\end{rmk}

\begin{rmk}
    The integral~\eqref{eq:EMint} becomes simply
    $- A'(x) \int_0^x \varphi_m(u) \, du$.
    Such integrals can be computed recursively in $m$,
    see Appendix~\ref{app:integrals}.
\end{rmk}

\subsubsection{First return time}
\label{sec:comp-return}

Let $K \geq 0$ and $\varphi(x, k) = (k + K) \varphi_m(x_k)$, where $\varphi_m$ is as
in~\eqref{eq:phi-n-cheb}.

Let $z = x^\gamma$, $t = z^{-1}$, $w = \tA(t)$.
In the complex plane, write
\[
    \varphi(x, s) = (s + K) \varphi_m(A^{-1}(A(x)+s))
\]
and
\[
    \psi(s)
    = \kappa(w) \frac{(s + K) \varphi_m(A^{-1}(A(x)+s))}{\kappa(w+s)}
    .
\]

Suppose that $\varkappa_1$ and $\nu$ are as in~\eqref{eq:nuka}, with $\varkappa_1 < \varkappa_0$.

\begin{lemma}
    \label{lem:tau_bound}
    Suppose that $\varkappa_1 < \varkappa_0$. Let
    $B = \frac{\varkappa_0 - \varkappa_1}{a_{-1}-C_1}$ and $C = \frac{\varkappa_1 \nu}{a_{-1}-C_1}$.
    If  $t$ is real with $t \ge r_1 + C$, then
    $\sup \{ |\psi(s + \zeta)| : s \geq 0, \, |\zeta| \leq \varkappa_1 (s + \nu) \} \leq C_\psi$,
    where
    \[
        C_\psi
        = \frac{a_{-1}+C_1}{a_{-1}-C_1}
        |t|^{1+1/\gamma}
        \Bigl[
            \frac{1+\varkappa_1}{B (t - C)^{1/\gamma}} + \frac{K + \varkappa_1 \nu}{(t - C)^{1+1/\gamma}}
        \Bigr]
        \|\varphi_m\|_{\cP_1}
        .
    \]
\end{lemma}

\begin{proof}
    As in the proof of Lemma~\ref{lem:kk}, with $t' = \tA^{-1}(w+s+\zeta)$, we have
    \[
        |\psi(s+\zeta)|
        = |K + s + \zeta| \biggl| \frac{\tA'(t)}{\tA'(t')} \biggr|
        \Bigl| \frac{t}{t'} \Bigr|^{1+1/\gamma}
        |\varphi_m(A^{-1}(w+s+\zeta))|
        ,
    \]
    where the derivative ratio is bounded by $\frac{a_{-1}+C_1}{a_{-1}-C_1}$, and
    it remains to bound $\frac{|K + s + \zeta|}{|t'|^{1+1/\gamma}}$.
    Also from the proof of Lemma~\ref{lem:kk}, $|t'| \geq t + B s - C$.
    Using $|K + s + \zeta| \le K + (1+\varkappa_1) s + \varkappa_1 \nu$
    and $(1+\varkappa_1)s \le \frac{1+\varkappa_1}{B}(Bs + t - C)$,
    \begin{align*}
        \frac{|K + s + \zeta|}{|t'|^{1+1/\gamma}}
        & \le \frac{K + (1 + \varkappa_1) s + \varkappa_1 \nu}{(B s + t - C)^{1+1/\gamma}}
        \le \frac{1+\varkappa_1}{B (Bs + t - C)^{1/\gamma}}
        + \frac{K + \varkappa_1 \nu}{(Bs + t - C)^{1+1/\gamma}}
        \\
        & \le \frac{1+\varkappa_1}{B (t - C)^{1/\gamma}}
        + \frac{K + \varkappa_1 \nu}{(t - C)^{1+1/\gamma}}
        .
    \end{align*}
    The result follows.
\end{proof}

Using the substitution $s = A(u) - A(x)$, write the integral~\eqref{eq:EMint} as
\[
    \int_0^\infty \psi(s) \, ds
    = - A'(x) \int_{0}^{x} [A(u) - A(x) + K] \varphi_m(u) \, du
    .
\]

\begin{lemma}
    \label{lem:log-cheb}
    Suppose that $z \in \cP_1$ is real. Then
    \begin{align*}
        \biggl|
        \int_{0}^{x}
        & [A(u) - A(x) + K] \varphi_m(u) \, du
        - \biggl[ \gamma a_\ell \int_0^x \varphi_m(u) \log u \, du
        \\
        &
        + [- A(x) + K] \int_0^x \varphi_m(u) \, du
        + \sum_{j = -1}^{n} a_j \int_0^x u^{\gamma j} \varphi_m(u) \, du
        \biggr] \biggr|
        \leq \frac{C_0 x^{n \gamma + 1}}{n \gamma + 1}
        ,
    \end{align*}
    where $C_0$ and $n$ are from Theorem~\ref{thm:A}.
    (See Appendix~\ref{app:integrals} for the computation of the integrals.)
\end{lemma}

\begin{proof}
    By linearity, the target integral separates into
    $[-A(x) + K] \int_0^x \varphi_m(u) \, du$ and $\int_0^x A(u) \varphi_m(u) \, du$.
    Approximate $A(u)$ by:
    \[
        A(u)
        = \hA(u^\gamma)
        \approx \hA_n(u^\gamma)
        = \gamma a_\ell \log u + \sum_{j=-1}^n a_j u^{j\gamma}
        .
    \]
    Then,
    \[
        \Bigl|
            \int_0^x A(u) \varphi_m(u) \, du
            - \int_0^x \hA_n(u^\gamma) \varphi_m(u) \, du
        \Bigr|
        \leq
        \int_0^x \bigl| A(u) - \hat{A}_n(u^\gamma) \bigr| \; |\varphi_m(u)| \, du
        .
    \]
    By Theorem~\ref{thm:A}, we have $|A(u) - \hat{A}_n(u^\gamma)| \leq C_0 u^{n\gamma}$ for $u \in [0, x]$.
    Furthermore, since $\varphi_m$ is a scaled Chebyshev polynomial,
    $\max_{0 \leq u \leq 1} |\varphi_m(u)| \leq 1$.
    So the right hand side above is bounded by
    $\int_0^x C_0 u^{n\gamma} \, du = \frac{C_0 x^{n\gamma+1}}{n\gamma+1}$.
    The result follows.
\end{proof}

\subsubsection{Lyapunov exponent}
\label{sec:comp-lyapunov}

In this section, we consider the observable
$\varphi(x, k) = \varphi_m(x_k) \log J_k(x)$,
where $\varphi_m$ is a Chebyshev polynomial~\eqref{eq:phi-n-cheb}.

Let $z = x^\gamma$, $t = z^{-1}$, $w = \tA(t)$.
In the complex plane, write $\varphi(x, s) = \varphi_m(A^{-1}(A(x)+s)) \log \frac{\kappa(w + s)}{\kappa(w)}$
and
\[  \psi(s)
    = \frac{
        \varphi_m(A^{-1}(A(x)+s))\log \frac{\kappa(w + s)}{\kappa(w)}
    }{
        \frac{\kappa(w+s)}{\kappa(w)}
    }
    .
\]

Suppose that $\varkappa_1$ and $\nu$ are as in~\eqref{eq:nuka}.

\begin{lemma}
    \label{lem:logJ_complex_k}
    Let $C = \frac{\varkappa_1 \nu}{a_{-1} - C_1}$.
    For real $t > r_1 + C$,
    \[
        \sup \{ |\psi(s + \zeta)| : s \geq 0, \, |\zeta| \leq \varkappa_1 (s + \nu) \} \leq C_\psi
    \]
    with
    \[
        C_\psi
        = M \sqrt{(\log M)^2 + \Theta^2}
        \, \|\varphi_m\|_{\cP_1}
        ,
    \]
    where $M = \frac{a_{-1} + C_1}{a_{-1} - C_1} \bigl( \frac{t}{t - C} \bigr)^{1 + 1/\gamma}$ and $\Theta = \arcsin\bigl(\frac{C_1}{a_{-1}}\bigr) + \bigl(1 + \frac{1}{\gamma}\bigr)\frac{\pi}{2}$.
\end{lemma}

\begin{proof}
    Let $Y = \frac{\kappa(w)}{\kappa(w+s+\zeta)}$.
    We aim to bound
    \[
        |Y \log(1/Y)| = |Y \log Y|
        = |Y| \sqrt{ [\log |Y|]^2 + [\arg Y]^2}
        .
    \]
    By Lemma~\ref{lem:kappa}, the modulus $r = |Y|$ is bounded by $M$.
    To bound the argument $\theta = \arg Y$, we use $Y = \frac{\tA'(t)}{\tA'(t')} \bigl( \frac{t}{t'} \bigr)^{1+1/\gamma}$, where $t' = \tA^{-1}(w+s+\zeta)$.
    Since $t$ is real, $\arg(t) = \arg(\tA'(t)) = 0$.
    By Lemma~\ref{lem:CA}, $|\arg(\tA'(t'))| \leq \arcsin(C_1/a_{-1})$.
    Furthermore, since $t' \in \tcP_1$, it lies in the right half-plane, which implies $|\arg(t')| < \pi/2$.
    Thus, $|\theta| \leq \Theta$.
    Note that the function $g(r) = r \sqrt{(\log r)^2 + \Theta^2}$ has derivative
    $g'(r) = \frac{(\log r)^2 + \log r + \Theta^2}{\sqrt{(\log r)^2 + \Theta^2}}$.
    Since $\Theta > \pi/2 > 1/2$, the numerator is strictly positive for all $r > 0$, making $g(r)$ strictly increasing on $(0, \infty)$.
    Substituting the upper bounds $r \leq M$ and $|\theta| \leq \Theta$ yields the result.
\end{proof}

\begin{lemma}
    \label{lem:uhuh}
    Suppose $x$ is real, such that $z \in \cP_1$. Then
    \begin{align*}
        \int_0^\infty \psi(s) \, ds
    & = A'(x) \biggl[
        \log \frac{-A'(x)}{\gamma} \int_0^x \varphi_m(u) \, du
        + (\gamma + 1) \int_0^x \varphi_m(u) \log u \, du
        \\
        &
        - \int_0^x \varphi_m(u) \log \bigl[
            a_{-1} - a_\ell u^\gamma - a_1 u^{2 \gamma} - 2 a_2 u^{3 \gamma}
            - \cdots - n a_n u^{(n + 1) \gamma}
        \bigr]\, du
    \biggr]
    + E
    ,
    \end{align*}
    where
    \[
        |E|
        \leq x A'(x) \log \Bigl[
            1 - \frac{C_0 (t - 1)^{-n}}{a_{-1} - C_1}
        \Bigr]
        .
    \]
\end{lemma}

\begin{proof}
    Using the substitution $s = A(u) - A(x)$, and $\frac{\kappa(w+s)}{\kappa(w)} = \frac{A'(u)}{A'(x)}$,
    write
    \begin{align*}
        \int_0^\infty \psi(s) \, ds
        & = A'(x) \biggl[
            \log(-A'(x)) \int_0^x \varphi_m(u) \, du
            - \int_0^x \varphi_m(u) \log(-A'(u)) \, du
        \biggr]
        \\
        & = A'(x) \biggl[
            \log \frac{-A'(x)}{\gamma} \int_0^x \varphi_m(u) \, du
            - (\gamma - 1) \int_0^x \varphi_m(u) \log u \, du
            \\
            & \qquad \qquad \qquad
            - \int_0^x \varphi_m(u) \log(-\hA'(u^\gamma)) \, du
        \biggr]
        .
    \end{align*}
    It remains to notice that $|\varphi_m| \leq 1$ on $[0,x]$ and by Lemmas~\ref{lem:CAn} and~\ref{lem:CA},
    \begin{align*}
        \bigl| \log(-\hA') - \log(-\hA_n') \bigr|
        & \leq - \log \Bigl[
            1 - \Bigl| \frac{\hA' - \hA'_n}{\hA'_n} \Bigr|
        \Bigr]
        = - \log \Bigl[
            1 - \Bigl| \frac{\tA' - \tA'_n}{\tA'_n} \Bigr|
        \Bigr]
        \\
        & \leq - \log \Bigl[
            1 - \frac{C_0 (t - 1)^{-n}}{a_{-1} - C_1}
        \Bigr]
        .
    \end{align*}
\end{proof}

\begin{rmk}
    \label{rmk:uhuh-int}
    All three integrals in Lemma~\ref{lem:uhuh} are computed explicitly.
    First two allow a closed-form expression, and for the third integral we approximate
    \[
        \log(a_{-1} - a_\ell u^\gamma - \cdots - n a_n u^{(n + 1) \gamma})
        \approx \log a_{-1} + c_1 u^\gamma + \cdots + c_\ell u^{\ell \gamma}
    \]
    as in Lemma~\ref{lem:log-poly}, which leaves us with a sum of integrals of $u^{k \gamma} \varphi_m(u)$
    which have a closed form expression, see Appendix~\ref{app:integrals}.
\end{rmk}

\section{Invariant density on \texorpdfstring{$[1/2,1]$}{[1/2,1]}}
\label{sec:id}

The main result of this section, Theorem~\ref{thm:hind}, is that the invariant
density $h$ of the induced transfer operator $\cL$ on $[1/2,1]$ is analytic
on an ellipse $A$ with foci at $1/2$ and $1$ and allows an approximation by
a finite sum of Chebyshev polynomials with an explicit error bound on $A$.


\subsection{Bernstein ellipses}

Bernstein ellipses with parameter $\rho$ by default mean
the standard Bernstein ellipses around $[-1,1]$ moved to $[1/2,1]$
by the linear map $z \mapsto z / 4 + 3 / 4$.

We are about to construct two Bernstein ellipses $A \subset C$
so that inverse branches of the induced map are analytic in $C$ and map
$C$ inside $A$.

Let $\theta_C = \min \{ \pi / (2 \gamma), \pi / 2 \}$ and
let $S_C$ be the sector $\{ z : |\arg(z)| < \theta_C \}$.

\begin{rmk}
    \label{rmk:NoWa}
    It is a direct verification that the principal branch of $f$ is analytic and invertible in $S_C$.
    The invertibility is a consequence of the Noshiro-Warschawski theorem: by the definition of $S_C$,
    any $z \in S_C$ satisfies $|\arg(z^\gamma)| < \pi / 2$, which ensures $\Re(z^\gamma) > 0$.
    Consequently, the derivative satisfies $\Re(f'(z)) = \Re(1+2^\gamma(\gamma+1)z^\gamma) > 1 > 0$.
    Since the sector $S_C$ is convex, this guarantees univalence.
    Furthermore, one can verify that $|f(z)| > |z|$ for all $z \in S_C$
    and $|\arg f(z)| > |\arg z|$ for all non-real $z \in S_C$,
    yielding $f^{-1}(S_C) \subset S_C$.
    As a corollary, all inverse branches of the induced map are analytic in $S_C$.
\end{rmk}

Let $C$ be a Bernstein ellipse properly contained in $S_C$, with parameter
\[
    \rho_C
    = \Bigl[ \sqrt{1 + 8 \sin^2(\theta_C)} + \sqrt{8} \sin(\theta_C) \Bigr]^{\eps_C}
    ,
\]
where $\eps_C \in (0,1)$. Choose e.g.\ $\eps = 7/8$ for the code.

Let $A$ be the Bernstein ellipse with parameter
\[
    \rho_A
    = \frac{s_A + \sqrt{s_A^2 - 4}}{2}
    , \quad \text{where} \quad
    s_A = 1 + \frac{1}{2} \Bigl( \rho_C + \frac{1}{\rho_C} \Bigr)
    .
\]

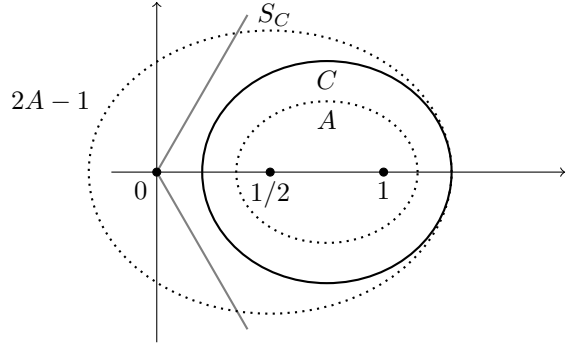
\[
\begin{tikzpicture}[scale=3]
  \def\mygamma{1.5}

  \pgfmathsetmacro{\tC}{min(90/\mygamma, 90)}
  \pgfmathsetmacro{\rC}{(sqrt(1 + 8*sin(\tC)^2) + sqrt(8)*sin(\tC))^0.875}
  \pgfmathsetmacro{\s}{1 + 0.5*(\rC + 1/\rC)}
  \pgfmathsetmacro{\rB}{(\s + sqrt(\s^2 - 4))/2}

  \pgfmathsetmacro{\aC}{0.125*(\rC + 1/\rC)} \pgfmathsetmacro{\bC}{0.125*(\rC - 1/\rC)}
  \pgfmathsetmacro{\aB}{0.125*(\rB + 1/\rB)} \pgfmathsetmacro{\bB}{0.125*(\rB - 1/\rB)}

  \draw[thick, gray] (0,0) -- (\tC:0.8) node[right, text=black] {$S_C$};
  \draw[thick, gray] (0,0) -- (-\tC:0.8);

  \draw[->, black] (-0.2, 0) -- (1.8, 0);
  \draw[->, black] (0, -0.75) -- (0, 0.75);

  \draw[thick]         (0.75, 0) ellipse (\aC cm and \bC cm) node[below] at (0.75, \bC) {$C$};
  \draw[thick, dotted] (0.75, 0) ellipse (\aB cm and \bB cm) node[below]  at (0.75, \bB) {$A$};
  \draw[thick, dotted] (0.5, 0) ellipse (2*\aB cm and 2*\bB cm) node[left]  at (0.5 - 2*\aB + 0.05, \bB) {$2A-1$};

  \filldraw (0.0,0) circle (0.5pt) node[below left] {$0$};
  \filldraw (0.5,0) circle (0.5pt) node[below] {$1/2$};
  \filldraw (1.0,0) circle (0.5pt) node[below] {$1$};
\end{tikzpicture}
\]

\begin{lemma}
    \label{lem:CtoB}
    Every inverse branch of the induced map maps $C$ into $A$.
    That is, $(f^{-k}(z) + 1) / 2 \in A$ for all $z \in C$
    and $k \geq 0$.
\end{lemma}

\begin{proof}
    We have to show that $f^{-k}(C) \subset 2 A - 1$ for all $k \geq 0$.
    Observe that $2 A - 1$ is an ellipse with foci at $0$ and $1$,
    and $C \subset 2 A - 1$ by construction.
    The iterations of $f^{-1}$ shrink both absolute value and
    argument closer to $0$, which guarantees that $f^{-k}(C)$
    stays in $2 A - 1$ for all $k \geq 1$.
\end{proof}

Similar to the ellipses $A,C$, by slightly tweaking the parameter
$\eps_C$, we choose slightly larger ellipses $A_+, C_+$.

 \subsection{Bounds on sums in \texorpdfstring{$A$}{A}}

Here we bound $S(x)$ in $A$ for all those $\varphi(x,k)$.

Let
\begin{equation}
    \label{eq:p}
    p = \bigl[ \inf_{z \in A^\gamma} \Re(z) \bigr]^{1/\gamma}
    .
\end{equation}

\begin{rmk}
    \label{rmk:aA}
    If $A^\gamma$ is convex, e.g.\ when $\gamma \approx 1$, then $p = 3/4 - a_A$ is the
    leftmost point of $A$.
\end{rmk}

\begin{lemma}
    \label{lem:Jp}
    $|J_k(x)| \geq J_k(p)$ for all $x \in A$.
\end{lemma}

\begin{proof}
    Recall that $f(x)$ is conjugate to $\hf(z) = z (1 + 2^\gamma z)^\gamma$,
    where $z = x^\gamma$. Let $z_j = \hf^{-j}(z)$, $c = (\gamma + 1) 2^\gamma$ and
    $q = -c^{-1}$, so that
    \begin{equation}
        \label{eq:JJk}
        J_k(z)
        = \prod_{j=1}^k |1 + c z_j|
        = \prod_{j=1}^k c |z_j - q|
        .
    \end{equation}

    We first show that for all $w = u + iv$ with $u > 0$, we have $\Re(\hf(w)) \leq \hf(u)$.
    By symmetry, we assume $v > 0$.

    Indeed, since $\hf$ is analytic, $\frac{\partial}{\partial v} \Re(\hf(w)) = -\Im(\hf'(w))$.
    Consider the argument of $\hf'(w) = c(w - q)(1 + 2^\gamma w)^{\gamma-1}$.
    Let $\theta_1 = \arg(w - q)$ and $\theta_2 = \arg(1 + 2^\gamma w)$.
    Since $q < 0$ and $w$ is in the first quadrant, $0 < \theta_2 < \theta_1 < \arg(w) \leq \frac{\pi}{2}$.
    Now, $\arg (\hf'(w)) = \theta_1 + (\gamma - 1)\theta_2$.
    If $\gamma \geq 1$, we have $0 < \theta_1 \leq \arg(\hf'(w)) < \gamma \theta_1 < \gamma \arg(w) \leq \frac{\pi}{2}$.
    If $\gamma \in (0, 1)$, we have $0 < \theta_1 - \theta_2 < \arg(\hf'(w)) < \theta_1 < \frac{\pi}{2}$.
    In both cases, $\arg(\hf'(w)) \in (0, \pi/2)$, which implies $\Im(\hf'(w)) > 0$.
    Therefore, $\frac{\partial}{\partial v} \Re(\hf(w)) < 0$ for $v > 0$, and $\Re(\hf(w)) \leq \hf(u)$
    as claimed.

    Applying this to the inverse map, if $\Re(z) \geq x_0$ and we write $\hf^{-1}(z) = u + iv$,
    then $x_0 \leq \Re(z) = \Re(\hf(u+iv)) \leq \hf(u)$. Since $\hf$ is increasing on the
    positive real line, $u \geq \hf^{-1}(x_0)$, meaning $\Re(\hf^{-1}(z)) \geq \hf^{-1}(x_0)$.

    By definition, $\Re(z) \geq p^\gamma$ for each $z \in A^\gamma$.
    Inductively, $\Re(z_j) \geq \hf^{-j}(p^\gamma) = p_j^\gamma$.
    Since $q$ is real and negative,
    \[
        |z_j - q|
        \geq \Re(z_j) - q
        \geq p_j^\gamma - q
        = |p_j^\gamma - q|
        .
    \]
    Substituting this lower bound into~\eqref{eq:JJk} completes the proof.
\end{proof}

We record two useful corollaries of Lemma~\ref{lem:Jp}:

\begin{cor}
    \label{cor:k_sum_A}
    \[
        \sup_{x \in A} \sum_{k=0}^\infty \frac{k + 1}{|J_k(x)|}
        \leq \sum_{k=0}^\infty \frac{k + 1}{J_k(p)}
        .
    \]
    The sums are finite if $\gamma < 1$.
\end{cor}
\begin{proof}
    The first inequality is immediate from $|J_k(x)| \geq J_k(p)$.
    The sums are finite when $\gamma < 1$ because $J_k(p) \sim k^{1+1/\gamma}$.
\end{proof}

\begin{cor}
    \label{cor:logJ_sum_A}
    \[
        \sup_{x \in A} \sum_{k=0}^\infty \Bigl| \frac{\log J_k(x)}{J_k(x)} \Bigr|
        \leq \sum_{k=0}^\infty \frac{\log J_k(p) + \pi}{J_k(p)}
        < \infty
        .
    \]
\end{cor}

\begin{proof}
    The map $f^{-k}$ is univalent on $A$, and $|\arg J_k(x)| \leq \pi$ for all $x \in A$.
    Consequently, $|\log J_k(x)| \leq \log |J_k(x)| + \pi$.
    By Lemma~\ref{lem:Jp}, $|J_k(x)| \geq J_k(p) \geq 1$.
    Hence
    \[
        \Bigl| \frac{\log J_k(x)}{J_k(x)} \Bigr|
        \leq \frac{\log|J_k(x)| + \pi}{|J_k(x)|}
        \leq \frac{\log J_k(p) + \pi}{J_k(p)}
        .
    \]
    The sum is finite because $J_k(p) \sim k^{1+1/\gamma}$.
\end{proof}

\subsection{Bounds in Hardy spaces}

The main function spaces now are the Hardy spaces of analytic functions
on Bernstein ellipses, e.g.\ $H_A$ with the supremum norm $\|\cdot\|_A$.
For operator norms we use the notation such as
$\|Q\|_{B \to A} = \sup_v \|Q v\|_A / \|v\|_B$.

\begin{lemma}
    \label{lem:LAA}
    $\cL \colon H_A \to H_A$ is a bounded operator with
    \[
        \| \cL \|_{A \to A} \leq \frac{1}{2} S(p)
        ,
    \]
    where $S$ is the sum~\eqref{eq:S} for $\varphi \equiv 1$
    and $p$ is as in~\eqref{eq:p}.
\end{lemma}

\begin{proof}
    We have to bound $|\cL v(x)|$ when $x \in A$ and $|v| \leq 1$.
    Write
    \[
        \cL v(x)
        = \frac{1}{2} \sum_{k = 0}^\infty \frac{v((x_k+1) / 2)}{J_k(x)}
        .
    \]
    Using Lemma~\ref{lem:Jp},
    \[
        |\cL v(x)|
        \leq \frac{1}{2} \sum_{k = 0}^\infty \frac{1}{|J_k(x)|}
        \leq \frac{1}{2} \sum_{k = 0}^\infty \frac{1}{J_k(p)}
        = \frac{1}{2} S(p)
        ,
    \]
    as required.
\end{proof}

\begin{rmk}
    \label{rmk:LAC}
    A similar bound holds for $\|\cL\|_{A \to C}$, albeit with a different point $p$.
\end{rmk}

\begin{lemma}
    \label{lem:Liscpt}
    $\cL$ is a compact operator on $H_A$.
\end{lemma}

\begin{proof}
    See~\cite[Proposition~2.7]{BJ07} or~\cite[Theorem~3.3]{BS20}.
\end{proof}

\begin{lemma}
    \label{lem:h}
    The operator $\cL \colon H_A \to H_A$ has $1$ as a maximal eigenvalue with multiplicity $1$,
    and a spectral gap.
    In particular, there exists a unique $h \in H_A$ so that $\cL h = h$ and $\int_{1/2}^1 h(x) \, dx = 1$.
\end{lemma}

\begin{proof}
    (See also~\cite[Proposition~5.4]{BJ07}.)

    Since the complex operator $\cL \colon H_A \to H_A$ is compact, its non-zero spectrum
    consists of isolated eigenvalues with finite multiplicity, thus it is a subset of the
    point spectrum of the real operator $\cL$ acting on Lipschitz functions on $[1/2, 1]$,
    and the multiplicities of the complex operator are at most the multiplicities of the
    real operator. The real operator has a single maximal eigenvalue $1$ and a spectral gap.

    It remains to show that $1$ is indeed an eigenvalue of the complex operator.
    This follows by contradiction: otherwise the spectral radius is strictly less than $1$,
    so $\cL^n v \to 0$ uniformly on $A$ for all $v \in H_A$, but
    $\int_{1/2}^1 (\cL^n v)(x) \, dx = 1$ for all $n$ when $v \equiv 2$.
\end{proof}

We index vectors and matrices starting from $0$.
Choose a large $N$.

Define $\tT_n(x) = T_n(4 x - 3)$ as scaled to $[1/2,1]$
Chebyshev polynomials.
Let
\begin{equation}
    \label{eq:tau}
    (\tau_0, \tau_1, \tau_2, \ldots )
    = (\tT_0 / 2, \tT_1, \tT_2, \ldots)
    .
\end{equation}
For an analytic function $v$, let $\bpi v$ denote its
interpolation on the linear space $\bW$
spanned by $\tau_0, \ldots, \tau_{N-1}$ given by the
discrete cosine transform.
That is,
\[
    \bpi v
    = \frac{2}{N} \sum_{k=0}^{N-1} \tau_k \sum_{n=0}^{N-1} \cos \Bigl[
        \frac{\pi (n + 1 / 2) k}{N}
    \Bigr]
    v (x_n)
    \quad \text{with} \quad
    x_n = \frac{1}{4} \cos \Bigl[ \frac{\pi (n + 1 / 2)}{N} \Bigr] + \frac{3}{4}
    .
\]
The action of $\bpi$ is not equivalent to truncating high index coordinates
in the basis $(\tau_0, \tau_1, \ldots)$. That is, generally
$\bpi(v_0, v_1, \ldots) \neq (v_0, \ldots, v_{N-1},0,\ldots)$.
Yet $\bpi$ is idempotent, and acts as an identity on $\bW$.

\begin{lemma}
    \label{lem:mnorm}
    For an operator $Q \colon \bW \to \bW$ with the corresponding matrix $Q$,
    \[
        \|Q\|_{A \to A}
        \leq K_{A,C} \, \| D_{\rho_C} Q D_{\rho_A}^{-1} \|_2
        ,
    \]
    where $\| \cdot \|_2$ is $L^2$ matrix norm (largest singular value),
    $D_\rho = \diag(d_0(\rho), d_1(\rho), \dots, d_{N-1}(\rho))$
    with $d_0(\rho) = 1/2$
    and $d_k(\rho) = \frac{1}{2} \sqrt{\rho^{2k} + \rho^{-2k}}$ for $k \geq 1$,
    and
    \[
        K_{A,C}
        = \biggl[ 1 + \sum_{k=1}^{N-1}
            \frac{(\rho_A^k + \rho_A^{-k})^2}{\rho_C^{2k} + \rho_C^{-2k}}
        \biggr]^{1/2}
        .
    \]
\end{lemma}

\begin{proof}
    Let $\|v\|_{\mathcal{H}^2(\rho)} = \|D_\rho v\|_2$
    be the Hardy-Hilbert norm.
    We use standard two bounds:
    \[
        \|v\|_{\mathcal{H}^2(\rho_A)}
        \leq \sup_{z \in A} |v(z)| = \|v\|_A
        \quad \text{and} \quad
        \|v\|_A
        \leq K_{A,C} \|D_{\rho_C} v\|_2
        .
    \]
    The desired result follows directly:
    \begin{align*}
        \|Qv\|_A
        & \leq K_{A,C} \|D_{\rho_C} Q v\|_2
        = K_{A,C} \| (D_{\rho_C} Q D_{\rho_A}^{-1}) D_{\rho_A} v \|_2
        \leq K_{A,C} \|D_{\rho_C} Q D_{\rho_A}^{-1}\|_2 \|D_{\rho_A} v\|_2
        \\
        &\leq K_{A,C} \|D_{\rho_C} Q D_{\rho_A}^{-1}\|_2 \|v\|_A
        .
    \end{align*}

    For completeness, we prove the two standard bounds.
    Parametrizing the boundary of $A$ as
    $z(\theta) = \frac{1}{8}(\rho_A e^{i\theta} + \rho_A^{-1} e^{-i\theta}) + \frac{3}{4}$ with $\theta \in [0, 2\pi)$,
    the Chebyshev polynomials take the form
    $\tT_k(z(\theta)) = \frac{1}{2}(\rho_A^k e^{ik\theta} + \rho_A^{-k} e^{-ik\theta})$.
    By Parseval's identity,
    \[
        \frac{1}{2\pi} \int_0^{2\pi} |v(z(\theta))|^2 \, d\theta
        = \frac{|v_0|^2}{4} + \sum_{k=1}^{N-1} \frac{|v_k|^2}{4} (\rho_A^{2k} + \rho_A^{-2k})
        = \|D_{\rho_A} v\|_2^2
        = \|v\|_{\mathcal{H}^2(\rho_A)}^2
        .
    \]
    The left hand side is bounded by $\|v\|_A^2$, and we obtain the first bound
    $\|v\|_A \geq \|v\|_{\mathcal{H}^2(\rho_A)}$.

    Next, using the triangle inequality and
    $\sup_{z \in A} |\tilde{T}_k(z)| = \frac{1}{2}(\rho_A^k + \rho_A^{-k})$, we bound for $z \in A$
    \[
        |v(z)|
        \leq \frac{|v_0|}{2} + \sum_{k=1}^{N-1} |v_k| \frac{\rho_A^k + \rho_A^{-k}}{2}
        .
    \]
    Applying the Cauchy-Schwarz inequality with the weights $\sqrt{\rho_C^{2k} + \rho_C^{-2k}}$,
    we obtain:
    \begin{align*}
        |v(z)|
        & \leq \frac{|v_0|}{2} \cdot 1
        + \sum_{k=1}^{N-1}
        |v_k| \frac{\sqrt{\rho_C^{2k} + \rho_C^{-2k}}}{2}
        \cdot
        \frac{\rho_A^k + \rho_A^{-k}}{\sqrt{\rho_C^{2k} + \rho_C^{-2k}}}
        \\
        & \leq \biggl[
            \frac{|v_0|^2}{4}
            + \sum_{k=1}^{N-1} |v_k|^2 \frac{\rho_C^{2k} + \rho_C^{-2k}}{4}
        \biggr]^{1/2}
        \biggl[
            1
            + \sum_{k=1}^{N-1} \frac{(\rho_A^k + \rho_A^{-k})^2}{\rho_C^{2k} + \rho_C^{-2k}}
        \biggr]^{1/2}
        = \|D_{\rho_C} v\|_2 \, K_{A,C}
        .
    \end{align*}
    Taking the supremum over $z \in A$ yields the second bound $\|v\|_A \leq K_{A,C} \|D_{\rho_C} v\|_2$.
\end{proof}

\begin{rmk}
    To compute $\| D_{\rho_C} Q D_{\rho_A}^{-1} \|_2$ in interval
    arithmetic, the intervals defining the matrix elements should not
    be too large. The bottom left corner of $D_{\rho_C} Q D_{\rho_A}^{-1}$ is expected
    to have intervals of width \emph{roughly}
    \[
        \rho_C^N
        \bigl[
            1 + x + \sqrt{x^2 + 2 x}
        \bigr]^N
        \biggl[
            \Bigl[ \frac{L}{\pi e \nu \varkappa_1} \Bigr]^{2 L}
            +
            e^{-M}
        \biggr]
        , \quad \text{where} \quad
        x = 2 r_1^{- 1 / \gamma}
        .
    \]
    We have to choose $L$, $\nu$ and $M$ to balance the exponential factor.
    Additional interval width may come from the evaluation of
    the Abel function.
\end{rmk}

\begin{lemma}
    \label{lem:prop3.2}
    \[
        \| I - \bpi \|_{C \to A}
        \leq \frac{ 8 (\rho_A / \rho_C)^{N-1} }{ \log (\rho_C / \rho_A)}
        ,
    \]
    and similar for other pairs of ellipses.
\end{lemma}

\begin{proof}
    See~\cite[Proposition~3.2]{VW25}.
\end{proof}

\begin{lemma}
    \label{lem:Q_apriori}
    Recall the ellipses $A_+$ and $C_+$.
    Let $Q$ be the $N \times N$ matrix representation of the operator
    $\bpi \cL$ acting on the finite-dimensional space $\bW$ spanned
    by $\tau_0, \dots, \tau_{N-1}$. Then
    \[
        |Q_{j,k}|
        \leq \frac{
            \bigl(
                \rho_{A_+}^k + \rho_{A_+}^{-k}
            \bigr)
            \bigl(
                \rho_{C_+}^{-j} + \rho_{C_+}^{j - 2N}
            \bigr)
        }{1 - \rho_{C_+}^{-2N}}
        \|\cL\|_{A_+ \to C_+}
        .
    \]
\end{lemma}

\begin{proof}
    Fix $k$ and suppose $v = \cL \tau_k = \sum_{m \geq 0} c_m \tau_m$.
    Since $v$ is analytic on $C_+$,
    we have $|c_m| \leq 2 \rho_{C_+}^{-m} \|v\|_{C_+}$.

    By the standard aliasing formula,
    see e.g.~\cite[Corollary 1]{X16},
    \[
        Q_{j,k}
        = c_j + \sum_{m = 1}^\infty (-1)^m (c_{2mN - j} + c_{2mN + j})
        .
    \]
    Then
    \[
        |Q_{j,k}|
        \leq 2 \|v\|_{C+} \biggl[
            \rho_{C_+}^{-j}
            + \sum_{m=1}^\infty \Bigl(
                \rho_{C_+}^{-(2mN-j)} + \rho_{C_+}^{-(2mN+j)}
            \Bigr)
        \biggr]
        = 2 \|v\|_{C+}
        \frac{
            \rho_{C_+}^{-j} + \rho_{C_+}^{-(2N-j)}
        }{
            1 - \rho_{C_+}^{-2N}
        }
        .
    \]
    It remains to notice that
    \[
        \|v\|_{C_+}
        = \|\cL \tau_k\|_{C_+}
        \leq \|\cL\|_{A_+ \to C_+} \|\tau_k\|_{A_+}
        \leq \frac{1}{2} \|\cL\|_{A_+ \to C+} (\rho_{A_+}^k + \rho_{A_+}^{-k})
        .
    \]
\end{proof}

Construct $\iota = (\iota_0, \iota_1, \ldots)$, $\iota_k \in \R$, from
the integrals of the basis Chebyshev polynomials
on $[1/2, 1]$:
\[
    \iota = \int_{[1/2,1]} (\tau_0(x), \tau_1(x), \ldots) \, dx
    .
\]
This way, for $v = \sum_k v_k \tau_k$, or $v = (v_0, v_1, \ldots)$,
\[
    \iota^T v
    = \int_{[1/2,1]} v(x) \, dx
    .
\]
Note also that $\iota_0 = 1/4$.

Let $u = (\iota_0^{-1}, 0, \ldots)$,
and define $(h_0, h_1, \ldots)  = (I - \cL + u \iota^T)^{-1} u$ and
$h = \sum_{k \geq 0} h_k \tau_k$.
\begin{lemma}
    $h$ is well defined, analytic on $A$, invariant under $\cL$,
    and $\int_{[1/2,1]} h(x) \, dx = 1$.
\end{lemma}

\begin{proof}
    By Lemma~\ref{lem:Liscpt}, $\cL$ is a compact operator on $H_A$. Since $u \iota^T$ is a rank-one operator, $\cL - u\iota^T$ is also compact. By the Fredholm alternative, $I - (\cL - u\iota^T)$ is invertible provided 1 is not an eigenvalue of $\cL - u\iota^T$.

    To show this, we suppose for contradiction that $(\cL - u\iota^T)v = v$ for some $v \neq 0$. Applying the integration functional $\iota^T$ to both sides, and using the properties $\iota^T \cL = \iota^T$ and $\iota^T u = 1$, we obtain $\iota^T v = 0$. Substituting $\iota^T v = 0$ back into the eigenvalue equation yields $\cL v = v$. By Lemma~\ref{lem:h}, $\cL$ has $1$ as a maximal eigenvalue with multiplicity $1$, thus $v$ must be proportional to the invariant density. However, $\iota^T v = 0$ forces $v = 0$, a contradiction. Thus, $(I - \cL + u\iota^T)$ is invertible on $H_A$, making $h = (I - \cL + u\iota^T)^{-1} u$ well defined. Since the constant vector $u \in H_A$, it follows that $h \in H_A$, ensuring $h$ is analytic on $A$.

    To verify the properties of $h$, we apply $\iota^T$ to the equation $(I - \cL + u\iota^T)h = u$:
    \[
        \iota^T h - \iota^T \cL h + (\iota^T u)(\iota^T h) = \iota^T u
        .
    \]
    Using $\iota^T \cL = \iota^T$ and $\iota^T u = 1$, the first two terms cancel, simplifying to $\iota^T h = 1$, which is exactly $\int_{[1/2,1]} h(x) \, dx = 1$.

    Finally, substituting $\iota^T h = 1$ back into the equation $(I - \cL + u\iota^T)h = u$ yields $\cL h = h$, completing the proof.
\end{proof}

\begin{lemma}
    \label{lem:uiota}
    \[
        \| u \iota^T (I - \bpi) \|_{A \to A}
        \leq \frac{8 \rho_A^{-(N-1)}}{\log \rho_A}
        .
    \]
\end{lemma}

\begin{proof}
    Let $U = [1/2,1]$ be the degenerate ellipse with parameter $\rho_U = 1$. Using Lemma~\ref{lem:prop3.2},
    $\|u\|_A \equiv 2$ and $\|\iota^T\|_{U \to \R} = 1/2$,
    \[
        \| u \iota^T (I - \bpi) \|_{A \to A}
        \leq \|u\|_A \|\iota^T\|_{U \to \R} \|I - \bpi\|_{A \to U}
        \leq \frac{8 \rho_A^{-(N-1)}}{\log \rho_A}
        .
    \]
\end{proof}

Define
\[
    \bh = (\bh_0, \ldots, \bh_{N - 1}, 0, \ldots)
    = \bigl(I - \bpi (\cL - u \iota^T) \bpi \bigr)^{-1} u
\]
and of course $\bh(x) = \sum_{k = 0}^{N - 1} \bh_k \tau_k(x)$.
Let
\begin{align*}
    \Delta = \cL - u \iota^T
    , \quad
    \bDelta = \bpi (\cL - u \iota^T) \bpi
    .
\end{align*}
Naturally,
\[
    \|\Delta\|_{A \to C}
    \leq 1 + \|\cL\|_{A \to C}
    .
\]
Using Lemmas~\ref{lem:prop3.2} and~\ref{lem:uiota},
we obtain an explicit bound
\[
    \| \Delta - \bDelta \|_{\bW \to A}
    \leq \| I - \bpi \|_{C \to A} \| \cL \|_{A \to C}
    .
\]

Let $R \colon \bW \to \bW$ be an approximate inverse of $I - \bDelta$.
That is, $R$ is an arbitrary matrix but for the following to work
it has to be close to $(I - \bDelta)^{-1}$.
Let $Z = I + R \bpi \Delta$ and $E = I - (I - \Delta) Z$.

\begin{lemma}
    \label{lem:I-Delta}
    If $\|E\|_{A \to A} < 1$, then
    \[
        \| (I - \Delta)^{-1} \|_{A \to A}
        \leq \frac{\|Z\|_{A \to A}}{1 - \|E\|_{A \to A}}
        .
    \]
    Further,
    \begin{align*}
        \|Z\|_{A \to A}
        & \leq 1 + \|R\|_{A \to A} \|\bpi\|_{C \to A} \|\Delta\|_{A \to C}
        , \\
        \|E\|_{A \to A}
        & \leq \| I - (I - \bDelta) R \|_{A \to A} \|\bpi\|_{C \to A}
        \|\Delta\|_{A \to C}
        + \|I - \bpi\|_{C \to A} \|\cL\|_{A \to C} \|Z\|_{A \to A}
        .
    \end{align*}
\end{lemma}

\begin{proof}
    Write $(I - \Delta) Z = I - E$. Then
    $(I - \Delta)^{-1} = Z (I - E)^{-1}$ and the bound on
    $\|(I - \Delta)^{-1}\|_{A \to A}$ is immediate.
    The bound on $\|Z\|_{A \to A}$ is also immediate, and
    the bound on $\|E\|_{A \to A}$ follows from
    \[
        E
        = \bigl( I - (I - \bDelta)R \bigr) \bpi \Delta + (I - \bpi) \Delta Z
        .
    \]
\end{proof}

We are now ready for the main result of this section:
\begin{thm}
    \label{thm:hind}
    \[
        \|h - \bh\|_{A}
        \leq C_{h,A}
        , \quad \text{where} \quad
        C_{h,A}
        = \| (I - \Delta)^{-1} \|_{A \to A}
        \| \Delta - \bDelta \|_{\bW \to A}
        \| \bh \|_A
         .
    \]
\end{thm}

\begin{proof}
    Write
    \[
        h - \bh
        = (I - \Delta)^{-1} (\Delta - \bDelta) \bh
        .
    \]
    The result follows.
\end{proof}

\begin{cor}
    \label{cor:hp}
    For real $x \in [1/2,1]$ and $k \geq 0$,
    \[
        |h^{(k)}(x) - \bh^{(k)}(x)|
        \leq \frac{k! C_{h,A}}{\eta^k}
        , \quad \text{where} \quad
        \eta
        = \dist \bigl( [1/2,1], \partial A \bigr)
        = \frac{\rho_A + \rho_A^{-1} - 2}{8}
        .
    \]
\end{cor}

\section{Return time and Lyapunov exponent}
\label{sec:app}

In this section we derive bounds that we use for our principal applications:
computation of mean return time to $[1/2,1]$ and Lyapunov exponent.

\subsection{Return time}
\label{sec:app-tau}

Let $\tau \colon [1/2,1] \to \{1,2,\ldots\}$,
$\tau(x) = \inf \{k \geq 1 : T^k(x) \in [1/2,1]\}$
be the first return time. We would like to compute
$\int_{1/2}^1 \tau(x) h(x) \, dx$.
For this integral to be finite, we assume that $\gamma < 1$.

For a function $v$, denote
\[
    S_v (x)
    = \sum_{k = 0}^\infty
    \frac{(k+1) v\bigl( \frac{x_k + 1}{2} \bigr)}{J_k(x)}
    .
\]
Following Section~\ref{sec:comp-return} for the computation of $S_v$ when
$v$ is a Chebyshev polynomial, then using Lemma~\ref{lem:prop3.2} for the aliasing error,
we get a good numerical approximation of $\bpi S_\bh$.

The computation of the mean return time is powered by:

\begin{prop}
    \label{prop:tau-approx}
    \[
        \Bigl|
            \int_{1/2}^1 \tau(x) h(x) \, dx
            - \frac{1}{2} \int_{1/2}^1 (\bpi S_\bh)(x) \, dx
        \Bigr|
        \leq \frac{1}{4} C_{h, A} S_1(1/2)
        + \frac{2 \|\bh\|_A S_1(p)}{\rho_A^{N-1} \log \rho_A}
        .
    \]
\end{prop}

The proof of Proposition~\ref{prop:tau-approx} takes the rest of this subsection.
Recall~\eqref{eq:L} and write
\[
    \int_{1/2}^1 \tau(x) h(x) \, dx
    = \int_{1/2}^1 \cL(\tau h)(x) \, dx
    \quad \text{and} \quad
    \cL(\tau h)(x)
    = \frac{1}{2} S_h(x)
    .
\]
The task is therefore reduced to bounding $|S_h - \bpi S_\bh|$ on $[1/2,1]$.
We bound $|S_h - S_\bh|$ and $|S_\bh - \bpi S_\bh|$ separately.

By Theorem~\ref{thm:hind}, $\|h - \bh\|_A \leq C_{h,A}$.
Using linearity of $S_v$ in $v$, the bound $|S_v(x)| \leq S_1(x) \sup_{[1/2,1]} |v|$
and that $\frac{1}{J_k(x)}$ is a decreasing function on $[1/2,1]$,
for all $x \in [1/2,1]$,
\begin{equation}
    \label{eq:tau-approx-1}
    |S_h(x) - S_\bh(x)|
    \leq C_{h, A} S_1(x)
    \leq C_{h, A} S_1(1/2)
    .
\end{equation}

\begin{lemma}
    \label{lem:tau-norm}
    $S_\bh$ is analytic in $A$ and
    $\| S_\bh\|_A \leq \|\bh\|_A S_1(p)$,
    where $p$ is as in~\eqref{eq:p}.
\end{lemma}

\begin{proof}
    By definition, $S_\bh(x) = \sum_{k = 0}^\infty \frac{(k+1)\bh(\frac{x_k+1}{2})}{J_k(x)}$.
    The analyticity is straightforward: the sum converges uniformly
    by Corollary~\ref{cor:k_sum_A}, and for each $x \in A$,
    the preimages $x_k$ are in the sector $S_C$,
    and $\frac{x_k+1}{2} \in A$, see Lemma~\ref{lem:CtoB}.
    Using $\bigl|\bh\bigl(\frac{x_k+1}{2}\bigr)\bigr| \leq \|\bh\|_A$ and Corollary~\ref{cor:k_sum_A},
    \[
        | S_\bh (x) |
        \leq \|\bh\|_A \sum_{k = 0}^\infty \frac{k+1}{|J_k(x)|}
        \leq \|\bh\|_A \sum_{k = 0}^\infty \frac{k+1}{J_k(p)}
        = S_1(p)
        ,
    \]
    as required.
\end{proof}

By Lemmas~\ref{lem:prop3.2} and~\ref{lem:tau-norm}, on $[1/2,1]$,
\begin{equation}
    \label{eq:tau-approx-2}
    |S_\bh - \bpi S_\bh|
    \leq \frac{ 8 \rho_A^{-(N-1)} }{ \log \rho_A}
    \|\bh\|_A S_1(p)
    .
\end{equation}
Assembling~\eqref{eq:tau-approx-1} and~\eqref{eq:tau-approx-2} yields
Proposition~\ref{prop:tau-approx}.

\subsection{Lyapunov exponent}
\label{sec:app-lyap}

We would like to compute $\lambda$, the Lyapunov exponent of $T$:
\[
    \lambda
    = \int_0^1 h(x) \log T'(x)\, dx
    \Big/
    \int_0^1 h(x) \, dx
    .
\]
With $\tau$ the first return time to $[1/2,1]$ as above, write
\[
    \lambda
    = ( \Lambda / 2 + \log 2)
    \Big/
    \int_{1/2}^1 \tau(x) h(x) \, dx
    , \quad \text{where} \quad
    \Lambda
    = \int_{1/2}^1 \sum_{k = 0}^\infty
    \frac{h \bigl(\frac{x_k + 1}{2}\bigr) \log J_k(x)}{J_k(x)} \, dx
    .
\]

The primary difficulty is therefore estimating $\Lambda$.
For a function $v$, denote
\[
    S_v(x)
    = \sum_{k = 0}^\infty
    \frac{v\bigl( \frac{x_k + 1}{2} \bigr) \log J_k(x)}{J_k(x)}
    .
\]
Following Section~\ref{sec:comp-lyapunov} for the computation of $S_v$ when
$v$ is a Chebyshev polynomial, then using Lemma~\ref{lem:prop3.2} for the aliasing error,
we get a good numerical approximation of $\bpi S_\bh$.

Similar to the computation of mean return time, the computation of Lyapunov exponent
is powered by the following proposition:

\begin{prop}
    \label{prop:Lambda-approx}
    \[
        \Bigl|
        \Lambda - \int_{1/2}^1 (\bpi S_\bh)(x) \, dx
        \Bigr|
        \leq
        \frac{C_{h, A}}{2}  R
        + \frac{4 \|\bh\|_A}{\rho_A^{N-1} \log \rho_A} R
        .
    \]
\end{prop}

Proof of Proposition~\ref{prop:Lambda-approx} takes the rest of this subsection;
it is very similar to that of Proposition~\ref{prop:tau-approx}.
Note that $\Lambda = \int_{1/2}^1 S_h(x) \, dx$,
so it is enough to bound $|S_h - S_\bh|$ and $|S_\bh - \bpi S_\bh|$ on $[1/2,1]$.

By Corollary~\ref{cor:logJ_sum_A}, $\sup_{x \in A} |S_1(x)| \leq R$ with
\[
    R
    = S_1(p) + \pi \sum_{k = 0}^\infty \frac{1}{J_k(p)}
    .
\]
Therefore,
\begin{equation}
    \label{eq:Lambda-approx-1}
    |S_h - S_\bh|
    \leq C_{h,A} R
\end{equation}

\begin{lemma}
    \label{lem:Lambda-norm}
    $S_\bh$ is analytic in $A$ and
    $\| S_\bh\|_A \leq \|\bh\|_A R$.
\end{lemma}

\begin{proof}
    Similar to the proof of Lemma~\ref{lem:tau-norm}.
\end{proof}

By Lemmas~\ref{lem:prop3.2} and~\ref{lem:Lambda-norm}, on $[1/2,1]$,
\begin{equation}
    \label{eq:Lambda-approx-2}
    |S_\bh - \bpi S_\bh|
    \leq \frac{ 8 \rho_A^{-(N-1)} }{ \log \rho_A}
    \|\bh\|_A R
    .
\end{equation}
Assembling~\eqref{eq:Lambda-approx-1} and~\eqref{eq:Lambda-approx-2} yields
Proposition~\ref{prop:Lambda-approx}.

\section{Invariant density on \texorpdfstring{$(0,1]$}{(0,1]}}
\label{sec:h01}

It is standard that
\begin{equation}
    \label{eq:Sh}
    h(x)
    = \frac{1}{2} \sum_{k=0}^\infty \frac{h \bigl( \frac{1 + x_k}{2} \bigr)}{J_k(x)}
    .
\end{equation}
Denote $z = x^\gamma$ and $w = A(x)$. Then
\begin{equation}
    \label{eq:ha}
    h(x)
    = \frac{1}{2} \sum_{k=0}^\infty \psi(k)
    \quad \text{with} \quad
    \psi(s)
    = \kappa(w) \frac{h \bigl( \frac{1 + A^{-1}(w + s)}{2} \bigr)}{\kappa(w + s)}
    .
\end{equation}
Using the approximation of $h(x)$ on $[1/2,1]$ and the Euler-Maclaurin
formula from Theorem~\ref{thm:EM}, it is now straightforward to accurately
compute $h(x)$ for small $x$.

In the rest of this section we derive and exploit a not-so-accurate
(yet still rigorous) approximation for not-so-small $x$:
\[
    h(x)
    \approx \frac{\gamma a_{-1}}{2} h(1/2) x^{-\gamma}
    .
\]
We repeat our analysis with a crude truncated Abel function as in
Section~\ref{sec:Abel} using $n = \ddn$ with $\ddn$ relatively small,
such as $8$: this allows for a larger petal $\cP_1$.
We mark the symbols with two dots on top to distinguish them
from those we use in the precise approximation.

Thus we choose $\ddvarkappa_1$ and $\ddnu$ as in~\eqref{eq:nuka}.
Let $\ddC, \ddR$ be as in Lemma~\ref{lem:kappa} and $\ddr_*$ as in
Corollary~\ref{cor:kappa}.
Using Theorem~\ref{thm:EM} with the number of derivatives $L = 1$,
for $\Re(z^{-1}) \geq \ddr_*$,
\begin{equation}
    \label{eq:hsum}
    \begin{aligned}
        h(x)
        & = \frac{1}{2} \int_0^\infty \psi(u) \, du
        + \frac{\psi(0)}{4}
        - \frac{B_{2}}{4} \psi'(0)
        + \frac{R_L}{2}
        \\
        & = - \frac{\gamma \kappa(w)}{2} \int_0^x h \Bigl( \frac{1 + v}{2} \Bigr) \, dv
        + \frac{1}{4} h \Bigl( \frac{1 + x}{2} \Bigr)
        - \frac{1}{24} \Bigl[
            \frac{h' \bigl( \frac{1 + x}{2} \bigr)}{2 A'(x)}
            - \frac{h \bigl( \frac{1 + x}{2} \bigr)A''(x)}{A'(x)^2}
        \Bigr]
        + \frac{R_L}{2}
        \\
        & = \frac{\gamma}{2} \frac{a_{-1}}{x^{\gamma + 1}}
        \int_0^x h \Bigl( \frac{1 + v}{2} \Bigr) \, dv
        + \cE(x)
        ,
    \end{aligned}
\end{equation}
where
\[
    |R_L|
    \leq \frac{3! C_\psi}{(2\pi)^3 \ddnu^2 \ddvarkappa_1^3}
    , \quad
    C_\psi
    \leq \|h\|_A \;
    \frac{a_{-1} + \ddC_1}{a_{-1} - \ddC_1}
    \ddR^{-(1 + 1/\gamma)}
    .
\]

\begin{rmk}
    We expect $\cE$ to be bounded, since
    $\kappa(w) = - \frac{a_{-1}}{x^{\gamma + 1}}
    + \frac{a_\ell}{x} + O(x^{\gamma - 1})$
    and $\int_0^x h \Bigl( \frac{1 + v}{2} \Bigr) \, dv = O(x)$.
\end{rmk}

For a function $q$, denote
$\| q \|_\# = \sup_{z \in \ddcP_1} |q((1 + x) / 2)|$.

\begin{rmk}
    \label{rmk:h'B}
    Recall the ellipse $A$.
    We make sure that $\ddr_1$ is large enough so that
    $(1 + x) / 2 \in A$ for all $z \in \ddcP_1$,
    with the distance to the boundary of $A$ at least
    \[
        \delta_\#
        = \dist \Bigl( \frac{1}{2}, \partial A \Bigr)
        - \frac{\ddr_1^{- 1 / \gamma}}{2}
        = \frac{\rho_A + \rho_A^{-1}}{8} - \frac{1}{4}
        - \frac{\ddr_1^{- 1 / \gamma}}{2}
        > 0
        .
    \]
    Then, using the bound on $\|h\|_A$ obtained with rigorous numerics,
    we bound $\|h^{(k)}\|_\# \leq k! \|h\|_A / \delta_\#^k$.
\end{rmk}

\begin{lemma}
    \label{lem:hE}
    With $\ddC, \ddR, \ddr_*$ as above,
    $|\cE(x)| \leq \Delta_*$ for $\Re(z^{-1}) \geq \ddr_*$, where
    \[
        \Delta_*
        = \Bigl[
            \frac{\gamma}{2} \ddr_1 \ddC_1 + \frac{1}{4}
            + \frac{C_\Delta}{24}
        \Bigr] \|h\|_A
        + \frac{\ddr_*^{-1 - 1/\gamma}}{48 \gamma(a_{-1} - \ddC_1)} \|h'\|_\#
        + \frac{|R_L|}{2}
    \]
    with $\ddC_1, \ddC_2$ from Lemmas~\ref{lem:CA}, \ref{lem:CA2} and
    \[
        C_\Delta
        = \frac{|\gamma-1|}{\gamma \ddr_* (a_{-1} - \ddC_1)}
        + \frac{2 a_{-1} + \ddC_2}{\ddr_* (a_{-1} - \ddC_1)^2}
        .
    \]
\end{lemma}

\begin{proof}
    From~\eqref{eq:hsum}, we decompose the error as
    $\cE(x) = \cE_1(x) + \cE_2(x) + \cE_3(x) + \frac{R_L}{2}$, where
    \begin{align*}
        \cE_1(x) &= -\frac{\gamma}{2} \Bigl( \kappa(w) + \frac{a_{-1}}{x^{\gamma+1}} \Bigr) \int_0^x h\Bigl(\frac{1+v}{2}\Bigr) \, dv , \\
        \cE_2(x) &= \frac{1}{4} h\Bigl(\frac{1+x}{2}\Bigr) , \\
        \cE_3(x) &= -\frac{1}{24} \Bigl[ \frac{h'\bigl(\frac{1+x}{2}\bigr)}{2 A'(x)} - \frac{h\bigl(\frac{1+x}{2}\bigr)A''(x)}{A'(x)^2} \Bigr] .
    \end{align*}

    For $\cE_1(x)$, using $t = x^{-\gamma}$ and $\kappa(w) = -\tA'(t)t^{1+1/\gamma}$, we have $\kappa(w) + a_{-1} x^{-(\gamma+1)} = (a_{-1} - \tA'(t))t^{1+1/\gamma}$.
    Since $\bigl|\int_0^x h(\frac{1+v}{2}) \, dv\bigr| \leq x \|h\|_A = |t|^{-1/\gamma} \|h\|_A$,
    using Lemma~\ref{lem:CA} we obtain
    \[
        |\cE_1(x)|
        \leq \frac{\gamma}{2} |t (\tA'(t) - a_{-1})| \|h\|_A
        \leq \frac{\gamma}{2} \ddr_1 \ddC_1 \|h\|_A
        .
    \]

    For $\cE_2(x)$, we immediately have $|\cE_2(x)| \leq \frac{1}{4}\|h\|_A$.

    For $\cE_3(x)$, applying Lemma~\ref{lem:CA},
    \[
        |A'(x)|
        = | \gamma \tA'(t) t^{1+1/\gamma}|
        \geq \gamma (a_{-1} - \ddC_1) \ddr_*^{1+1/\gamma}
        .
    \]
    This bounds the first term of $\cE_3(x)$ by
    $\frac{\ddr_*^{-1-1/\gamma}}{48 \gamma(a_{-1} - \ddC_1)} \|h'\|_\#$.

    To bound the second term of $\cE_3(x)$, we express $A''(x)/A'(x)^2$ in terms of $z = x^\gamma$:
    \[
        \frac{A''(x)}{A'(x)^2} = \frac{\hA''(z)}{\hA'(z)^2} + \frac{\gamma - 1}{\gamma z \hA'(z)} = \frac{z^3 \hA''(z)}{t \tA'(t)^2} + \frac{1 - \gamma}{\gamma t \tA'(t)} .
    \]
    Applying Lemma~\ref{lem:CA} and Lemma~\ref{lem:CA2}, we bound
    $|\tA'(t)| \geq a_{-1} - \ddC_1$ and $|z^3 \hA''(z)| \leq 2a_{-1} + \ddC_2$.
    Thus,
    \[
        \Bigl| \frac{A''(x)}{A'(x)^2} \Bigr|
        \leq \frac{2a_{-1} + \ddC_2}{\ddr_*(a_{-1} - \ddC_1)^2}
        + \frac{|\gamma - 1|}{\gamma \ddr_*(a_{-1} - \ddC_1)}
        = C_\Delta
        .
    \]
    This bounds the second term of $\cE_3(x)$ by $\frac{C_\Delta}{24} \|h\|_A$.

    Finally, summing the bounds for $|\cE_1|, |\cE_2|, |\cE_3|$ and $|\frac{R_L}{2}|$
    yields the stated bound $\Delta_*$.
\end{proof}

Denote $\ddx_* = 2^{-1 - \frac{1}{2 \gamma}} \, \ddr_*^{-1/\gamma}$ and
$\theta_* = 1 / \sin \bigl( \min \bigl\{ \frac{\pi}{2}, \frac{\pi}{4 \gamma} \bigr\} \bigr)$.

\begin{cor}
    \label{cor:hE}
    For real $x \leq \ddx_*$,
    \[
        |\cE'(x)| \leq \frac{\Delta_* \theta_*}{x}
        , \quad
        |\cE''(x)| \leq \frac{2 \Delta_* \theta_*^2}{x^2}
        \quad \text{and} \quad
        |\cE'''(x)| \leq \frac{6 \Delta_* \theta_*^3}{x^3}
        .
    \]
\end{cor}

\begin{proof}
    Let $\alpha = \min\bigl\{ \frac{\pi}{2}, \frac{\pi}{4\gamma} \bigr\}$,
    so that $\theta_* = 1 / \sin \alpha$. Consider the complex disk $D$
    centered at $x$ with radius $\sigma = x / \theta_* = x \sin \alpha$.

    For any $\zeta \in D$, simple geometry implies
    $|\arg \zeta| \leq \alpha \leq \frac{\pi}{4\gamma}$, hence
    $\cos(\gamma \arg \zeta) \geq \cos(\frac{\pi}{4}) = 2^{-1/2}$.
    Moreover, $|\zeta| \leq x + \sigma \leq 2x$.
    This yields the bound
    \[
        \Re(\zeta^{-\gamma}) = |\zeta|^{-\gamma} \cos(\gamma \arg \zeta) \geq (2x)^{-\gamma} 2^{-1/2} = 2^{-(\gamma + 1/2)} x^{-\gamma} .
    \]
    Since $x \leq \ddx_* = 2^{-\frac{\gamma + 1/2}{\gamma}} \ddr_*^{-1/\gamma}$,
    we obtain $\Re(\zeta^{-\gamma}) \geq \ddr_*$.

    This ensures that $D$ is contained in the domain where
    Lemma~\ref{lem:hE} guarantees $|\cE(\zeta)| \leq \Delta_*$.
    Applying Cauchy's bound $|\cE^{(n)}(x)| \leq n! \Delta_* / \sigma^n$
    yields the desired results.
\end{proof}

\begin{lemma}
    \label{lem:h-near-zero}
    Let $M_0 = \frac{\gamma a_{-1}}{2} h(1/2)$ and $M_1 = \frac{\gamma a_{-1}}{8} h'(1/2)$. Write
    \begin{align*}
        h(x)
        &= M_0 x^{-\gamma} + M_1 x^{1-\gamma}
        + E_0(x) + \cE(x)
        , \\
        h'(x)
        &= -\gamma M_0 x^{-\gamma - 1}
        + (1-\gamma) M_1 x^{-\gamma}
        + E_1(x) + \cE'(x)
        , \\
        h''(x)
        &= \gamma(\gamma+1) M_0 x^{-\gamma - 2}
        - \gamma(1-\gamma) M_1 x^{-\gamma - 1}
        + E_2(x) + \cE''(x)
        , \\
        h'''(x)
        &= - \gamma(\gamma+1)(\gamma + 2) M_0 x^{-\gamma - 3}
        + \gamma(\gamma+1)(1-\gamma) M_1 x^{-\gamma - 2}
        + E_3(x) + \cE'''(x)
        .
    \end{align*}
    For all real $x \leq \ddx_*$,
    we have $|E_0(x)| \leq D_0 x^{2 - \gamma}$,
    $|E_1(x)| \leq D_1 x^{1 - \gamma}$, $|E_2(x)| \leq D_2 x^{-\gamma}$,
    and $|E_3(x)| \leq D_3 x^{-1 - \gamma}$, where
    \[
        D_0 = \frac{\gamma a_{-1}}{48} \|h''\|_\#
        , \quad
        D_1 = (\gamma + 4) D_0
        , \quad
        D_2 = (\gamma^2 + 9\gamma + 14) D_0
        ,
    \]
    \[
        D_3 = (\gamma^3 + 15\gamma^2 + 56\gamma + 48) D_0
        .
    \]
\end{lemma}

\begin{proof}
    Let $I(x) = \int_0^x h\bigl( \frac{1+v}{2} \bigr) \, dv$. Write
    \[
        I(x)
        = h(1/2) x + \frac{1}{4} h'(1/2) x^2 + E_I(x)
        .
    \]
    Note that $E_I'''(x) = I'''(x) = \frac{1}{4} h''(\frac{1+x}{2})$ is bounded by $\frac{1}{4}\|h''\|_\#$.
    Successively integrating from $0$ to $x$, we obtain
    $|E_I^{(k)}(x)| \leq \frac{1}{4(3-k)!} \|h''\|_\# x^{3-k}$ for $k=0,1,2$.

    From~\eqref{eq:hsum}, we obtain the target expression for $h(x)$ with
    $E_0(x) = \frac{\gamma a_{-1}}{2} x^{-\gamma-1} E_I(x)$, and
    the bound for $E_I(x)$ gives $|E_0(x)| \leq D_0 x^{2-\gamma}$.

    The higher-order bounds $|E_k(x)| \leq D_k x^{2-\gamma-k}$ for $k=1,2,3$ follow by setting
    $E_k(x) = E_0^{(k)}(x)$ and directly differentiating the product $\frac{\gamma a_{-1}}{2} x^{-\gamma-1} E_I(x)$,
    bounding each resulting term with the established estimates for $E_I^{(k)}(x)$.
\end{proof}

\begin{cor}
    \label{cor:h-near-zero}
    Let $M_0$ and $M_1$ be as in Lemma~\ref{lem:h-near-zero}. Set
    \begin{align*}
        \delta_0
        & = |M_1| \ddx_* + D_0 \ddx_*^2 + \Delta_* \ddx_*^\gamma
        &
        \delta_1
        & = |1-\gamma| |M_1| \ddx_* + D_1 \ddx_*^2 + \Delta_* \theta_* \ddx_*^\gamma
        \\
        \delta_2
        & = \gamma |1-\gamma| |M_1| \ddx_* + D_2 \ddx_*^2 + 2 \Delta_* \theta_*^2 \ddx_*^\gamma
        &
        \delta_3
        & = \gamma(\gamma+1) |1-\gamma| |M_1| \ddx_* + D_3 \ddx_*^2 + 6 \Delta_* \theta_*^3 \ddx_*^\gamma
    \end{align*}
    Let
    \[
        F_1^- = \frac{\gamma M_0 - \delta_1}{M_0 + \delta_0}
        , \quad
        F_2^- = \frac{\gamma(\gamma+1) M_0 - \delta_2}{M_0 + \delta_0}
        , \quad
        F_3^- = \frac{\gamma(\gamma+1)(\gamma+2) M_0 - \delta_3}{M_0 + \delta_0}
        ,
    \]
    \[
        F_1^+ = \frac{\gamma M_0 + \delta_1}{M_0 - \delta_0}
        , \quad
        F_2^+ = \frac{\gamma(\gamma+1) M_0 + \delta_2}{M_0 - \delta_0}
        , \quad
        F_3^+ = \frac{\gamma(\gamma+1)(\gamma+2) M_0 + \delta_3}{M_0 - \delta_0}
        .
    \]
    Assuming that $\ddx_*$ is small enough so that $M_0 > \delta_0$,
    for all real $x \leq \ddx_*$ and $k=1,2,3$,
    \[
        \frac{F_k^-}{x^k}
        \leq (-1)^k \frac{h^{(k)}(x)}{h(x)}
        \leq \frac{F_k^+}{x^k}
        .
    \]
\end{cor}

\begin{rmk}
    In practice, our rigorous numerics ensure that $F_k^- > 0$ for $k=1,2,3$, and
    that $F_k^-$ and $F_k^+$ are \emph{close}.
    This becomes important later, in Section~\ref{sec:cond}.
\end{rmk}

\begin{proof}[Proof of Corollary~\ref{cor:h-near-zero}]
    By Lemma~\ref{lem:h-near-zero},
    $ x^\gamma h(x) = M_0 + \bigl[ M_1 x + x^\gamma E_0(x) + x^\gamma \cE(x) \bigr]$.
    Write for $k=0,\ldots,3$,
    \[
        x^{\gamma+k} (-1)^k h^{(k)}(x)
        = H_k + \Delta_{k}(x)
        ,
    \]
    where $H_0 = M_0$, $H_1 = \gamma M_0$, $H_2 = \gamma(\gamma+1)M_0$,
    and $H_3 = \gamma(\gamma+1)(\gamma+2)M_0$.

    We apply the triangle inequality to $\Delta_k$ for $x \leq \ddx_*$,
    substituting the bounds $|E_k(x)| \leq D_k x^{2-\gamma-k}$ (Lemma~\ref{lem:h-near-zero})
    and $|\cE^{(k)}(x)| \leq k! \Delta_* (\theta_*/x)^k$ (Corollary~\ref{cor:hE}).
    This way,
    \[
        |\Delta_0 (x)|
        = \bigl| M_1 x + x^\gamma E_0(x) + x^\gamma \cE(x) \bigr|
        \leq |M_1| \ddx_* + D_0 \ddx_*^2 + \Delta_* \ddx_*^\gamma = \delta_0 .
    \]
    Similarly, $|\Delta_k(x)| \leq \delta_k$ for $k=1,2,3$.

    Note that $M_0 > \delta_0$ and $x^\gamma h(x) \in [M_0 - \delta_0, M_0 + \delta_0]$,
    in particular $x^\gamma h(x)$ is positive. Standard interval arithmetic on the ratio
    $\frac{H_k \pm \delta_k}{M_0 \mp \delta_0}$ results in the target bounds on
    $(-1)^k \frac{h^{(k)}(x)}{h(x)}$.
\end{proof}

\section{Conditions for memory loss}
\label{sec:cond}

In this section we prepare the proof that on $(0,1]$:
\begin{equation}
    \label{eq:hh}
    \Bigl(\frac{h'}{h}\Bigr)' > 0
    \quad \text{and} \quad
    \Bigl(\frac{h''}{h}\Bigr)' < 0
    .
\end{equation}
These properties are important for our results on memory loss.
We do not know how to prove~\eqref{eq:hh} analytically, so we set up a
verification using rigorous numerics.

On $(0, \ddx_*]$ we prove~\eqref{eq:hh} using Corollary~\ref{cor:h-near-zero},
and on $[1/2,1]$ using the Chebyshev approximation $h \approx \bh$ in the ellipse $A$,
both are pretty straightforward. It remains to cover the interval $(\ddx_*, 1/2)$.
This does not look hard: for every $x \in (0, 1/2)$ the value of $h(x)$ can be computed
recursively using the relation
\begin{equation}
    \label{eq:h-inv}
    h(f(x))
    = \frac{1}{2}h\Bigl(\frac{1+f(x)}{2}\Bigr) + \frac{h(x)}{f'(x)}
    .
\end{equation}
But numerically such a computation leads to a major loss of precision and is not practical,
so we need a more carefully chosen approach.

Thus, to cover the interval $(\ddx_*, 1/2)$, we verify numerically that
$(-1)^k h^{(k)} > 0$ on $[1/2,1]$ for $k=0,\ldots, 3$,
and recursively apply the following Lemma. Starting from a \emph{good}
behavior of $h$ on $(0, \ddx_*]$ (as in Corollary~\ref{cor:h-near-zero}), is
expands bounds on derivatives of $h$ from an interval
$[p,q] \subset (0,1/2]$ to $[f(p), f(q)]$.

\begin{lemma}
    \label{lem:hmon}
    Assume that $(-1)^k h^{(k)} > 0$ on $[1/2,1]$ for $k=0,\ldots, 4$.
    Suppose that on an interval $[p,q] \subset (0,1/2]$,
    \[
        (-1)^k h^{(k)}(x) \geq C_k
        \quad \text{for} \quad
        k = 0,\ldots,3
        \quad \text{and} \quad -h'(x) \leq C'
        ,
    \]
    where $C_k$ and $C'$ are positive constants.
    Then $(-1)^k h^{(k)} \geq D_k$ for $k=0,\ldots,3$ on $[f(p), f(q)]$, where $D_k$
    are as follows.

    Let $[f_j, F_j] = f^{(j)}([p,q])$. Note that $f_1, f_2 > 0$.
    For $j = 3,4$, let
    \[
        f_j^+ = \max(0, f_j)
        , \quad
        F_j^+ = \max(0, F_j)
        , \quad
        f_j^- = \max(0, -F_j)
        , \quad
        F_j^- = \max(0, -f_j)
        .
    \]
    Then
    \begin{align*}
        D_0
        & = \frac{C_0}{F_1}
        + \frac{1}{2} h\Bigl(\frac{1+f(q)}{2}\Bigr)
        \\
        D_1
        & = \frac{C_1}{F_1^2} + \frac{C_0 f_2}{F_1^3}
        - \frac{1}{4} h'\Bigl(\frac{1+f(q)}{2}\Bigr)
        \\
        D_2
        & = \frac{C_2}{F_1^3} + \frac{3 C_1 f_2}{F_1^4} + \frac{3 C_0 f_2^2}{F_1^5}
        + \frac{C_0 f_3^-}{F_1^4} - \frac{h(p) F_3^+}{f_1^4}
        + \frac{1}{8} h''\Bigl(\frac{1+f(q)}{2}\Bigr)
        \\
        D_3
        & = \frac{C_3}{F_1^4} + \frac{6 C_2 f_2}{F_1^5} + \frac{15 C_1 f_2^2}{F_1^6}
        + \frac{15 C_0 f_2^3}{F_1^7} + \frac{4 C_1 f_3^- + C_0 f_4^+}{F_1^5}
        \\
        & \qquad \qquad + \frac{10 C_0 f_2 f_3^-}{F_1^6}
        - \frac{4 C' F_3^+ + h(p) F_4^-}{f_1^5} - \frac{10 h(p) F_2 F_3^+}{f_1^6}
        - \frac{1}{16} h'''\Bigl(\frac{1+f(q)}{2}\Bigr)
        .
    \end{align*}
\end{lemma}

\begin{proof}
    We give a sketch of a proof.
    Differentiate~\eqref{eq:h-inv} with respect to $y = f(x)$, using $\frac{dx}{dy} = \frac{1}{f'(x)}$:
    \begin{align*}
        h(y)
        & = \frac{h(x)}{f'(x)} + \frac{1}{2}h\Bigl(\frac{1+y}{2}\Bigr)
        , \\
        -h'(y)
        & = \frac{-h'(x)}{(f'(x))^2} + \frac{h(x)f''(x)}{(f'(x))^3} - \frac{1}{4}h'\Bigl(\frac{1+y}{2}\Bigr)
        .
    \end{align*}
    Since $(-1)^k h^{(k)} > 0$ on $[1/2,1]$, both $h \bigl( \frac{1+y}{2} \bigr)$
    and $-h' \bigl( \frac{1+y}{2} \bigr)$ attain their minimum at the right endpoint $y = f(q)$.
    The bounds $D_0$ and $D_1$ follow straightforwardly using the bounds for
    $(-1)^k h^{(k)}(x)$ and $f', f''$ on $[p,q]$.

    The bounds $D_2$ and $D_3$ are obtained similarly, taking further derivatives
    with respect to $y$ and applying the assumed bounds.
\end{proof}

\appendix

\section*{Appendix}

\section{Basic lemmas}
\label{app:basic}

\begin{lemma}
    \label{lem:Imt}
    Let $\gamma > 0$ and $b = 2^\gamma$, and $g(t) = t(1 + b / t)^{-\gamma}$.
    For $t \in \C$ with $\Re(t) \geq b (\gamma + 1)$,
    \begin{equation*}
        |\Im(g(t))| \leq |\Im(t)|
        ,
    \end{equation*}
    with the equality only if $\Im(t) = 0$.
\end{lemma}

\begin{proof}
    Assume $t = u + iv \in \tcP$ with $v > 0$. The case $v < 0$ is similar, and the case $v = 0$ is trivial.
    Let $R, \varphi$ be the polar coordinates of $\overline{1 + b / t}$,
    so that $1 + b / t = R e^{-i \varphi}$.
    Then
    \[
        \frac{u}{v}
        = \frac{R \cos \varphi - 1}{R \sin \varphi}
        .
    \]
    Write $g(t) = t R^{-\gamma} e^{i \gamma \varphi}$. The inequality $\Im(g(t)) < v$ is equivalent to
    \[
        u \sin(\gamma \varphi) + v \cos(\gamma \varphi)
        < v R^\gamma
        .
    \]
    Dividing by $v$ and substituting $u/v$, this condition becomes $H(R) > 0$, where
    \[
        H(R)
        = R^{\gamma + 1} \sin \varphi - R \sin((\gamma + 1) \varphi) + \sin(\gamma \varphi).
    \]
    Observe that $R \sin \varphi = \frac{b v}{u^2 + v^2} \leq \frac{b}{2 u}$.
    Since $R > 1$ and $u \geq b(\gamma+1)$,
    \[
        \sin \varphi
        < \frac{b}{2 u} < \frac{1}{2 (\gamma + 1)}
        ,
    \]
    and using $\frac{\arcsin x}{x} \leq 2 \arcsin \frac{1}{2}$ for $x \in (0,1/2]$,
    \[
        2 (\gamma + 1) \varphi
        < 2 (\gamma + 1) \arcsin \Bigl[ \frac{1}{2(\gamma+1)} \Bigr]
        \leq \frac{\pi}{3}
        .
    \]
    In particular, $\varphi$ and $(\gamma + 1) \varphi$ are in $(0, \pi / 6]$.
    Hence
    \[
        H(1)
        = 2\sin(\varphi / 2) [\cos(\varphi / 2) - \cos((\gamma + 1/2) \varphi)]
        > 0.
    \]
    Finally, $H(R) > 0$ for all $R \geq 1$ follows from
    \[
        H'(R)
        = (\gamma + 1) R^\gamma \sin \varphi - \sin( (\gamma + 1) \varphi)
        \geq (\gamma + 1) R^\gamma \sin \varphi - (\gamma + 1) \sin \varphi
        \geq 0
        .
    \]
\end{proof}

In the following lemma, we are mostly interested in the case where
$a_k = 0$ for all large $k$.

\begin{lemma}
    \label{lem:log-poly}
    Let $f(x) = a_0 + \sum_{k=1}^\infty a_k x^k$ with $a_0 > 0$.
    Assume that $\sum_{k=1}^\infty |a_k| \delta^k < a_0$
    for some $\delta > 0$.
    Then for all $n \geq 1$ and $x \in \C$ with $|x| \leq \delta$,
    \[
        \log f(x)
        = \log a_0 + \sum_{k = 1}^n c_k x^k + E_n(x)
        ,
    \]
    where
    \[
        c_k
        = \frac{a_k}{a_0} - \frac{1}{k}
        \sum_{j = 1}^{k - 1} (k - j) c_{k - j}
        \frac{a_j}{a_0}
        , \qquad
        \tc_k
        = \frac{|a_k|}{a_0} + \frac{1}{k}
        \sum_{j = 1}^{k - 1} (k - j) \tc_{k - j} \frac{|a_j|}{a_0}
        ,
    \]
    and
    \[
        |E_n(x)|
        \leq -\log \Bigl[
            1 - \sum_{j=1}^\infty \frac{|a_j|}{a_0} \delta^j
        \Bigr]
        - \sum_{k=1}^n \tc_k \delta^k
        ,
    \]
    and
\end{lemma}

\begin{proof}
    Without loss of generality, assume $a_0 = 1$.
    Let $\log f(x) = \sum_{k = 1}^\infty c_k x^k$.
    Taking the derivative gives $f'(x) = f(x) \frac{d}{dx} \log f(x)$.
    Writing both sides as formal power series and equating the coefficients
    yields the recurrence for $c_k$.

    To bound the tail $E_n(x) = \sum_{k=n+1}^\infty c_k x^k$,
    we use the majorant series
    $\tf(x) = 1 - \sum_{k=1}^\infty |a_k| x^k$.
    The Maclaurin coefficients of
    $-\log \tf(x) = \sum_{k=1}^\infty \tc_k x^k$
    are governed by the recurrence in the statement of the lemma;
    in particular, $\tc_k \geq |c_k|$.

    For $|x| \leq \delta$,
    \[
        |E_n(x)|
        \leq \sum_{k=n+1}^\infty |c_k| \delta^k
        \leq \sum_{k=n+1}^\infty \tc_k \delta^k
        = -\log \tf(\delta) - \sum_{k=1}^n \tc_k \delta^k
        .
    \]
\end{proof}

\section{Computation of Integrals}
\label{app:integrals}

In this appendix, we describe the algorithms we use in our code for computing definite
integrals of the scaled to $[0,1]$ Chebyshev polynomials, $\varphi_m(u) = T_m(2u - 1)$,
over the interval $[0, x]$.

\subsection{Fractional power integrals}

To evaluate integrals of the form
\[
    I_{m, \beta}
    = \int_0^x u^\beta \varphi_m(u) \, du
\]
for $\beta > -1$, we use the relation
$2(2u-1)\varphi_m(u) = \varphi_{m+1}(u) + \varphi_{m-1}(u)$.
Integrating by parts, it yields the recurrence
\[
    I_{m, \beta}
    = \frac{
        2\beta I_{m-1, \beta} + (m - 3 - \beta) I_{m-2, \beta}
        - 4(1 - x) x^{\beta+1} \varphi_{m-1}(x)
    }{\beta + m + 1}
    .
\]
This allows a computation for $m \geq 2$ after explicitly computing
\[
    I_{0, \beta} = \frac{x^{\beta+1}}{\beta+1}
    \quad \text{and} \quad
    I_{1, \beta} = \frac{2x^{\beta+2}}{\beta+2} - \frac{x^{\beta+1}}{\beta+1}
    .
\]

\subsection{Logarithmic integrals}
For logarithmic integrals of the form
\[
    L_m = \int_0^x \varphi_m(u) \log u \, du ,
\]
we apply a similar technique. By differentiating the recurrence for $I_{m, \beta}$ with respect to $\beta$ and evaluating it at $\beta = 0$, we obtain a recurrence coupling $L_m$ with the non-logarithmic integrals $I_{m, 0}$:
\[
    L_m = \frac{
        (m-3) L_{m-2} - I_{m, 0} + 2 I_{m-1, 0} - I_{m-2, 0}
        - 4x (1 - x) \log(x) \varphi_{m-1}(x)
    }{m+1}
    .
\]
Now $L_m$ for $m \geq 2$ is computed recursively, starting from
\[
    L_0 = x \log x - x
    \quad \text{and} \quad
    L_1 = (x^2 - x) \log x - \frac{1}{2}x^2 + x
    .
\]

\subsection*{Acknowledgements}

The authors are grateful to Great Bay University, and in particular Cece (Qian Zhang),
for warm hospitality and support, and to Yiwei Zhang (Anhui University of Science and Technology)
for helpful comments.

CW was supported by the Australian Research Council Discovery Early Career Award DE260101080.

\end{document}